\documentclass[11pt]{article}
\usepackage[T1]{fontenc}
\usepackage{latexsym,amssymb,amsmath,amsfonts,amsthm}
\usepackage{graphics}
\usepackage[section]{placeins}
\usepackage{graphicx}
\usepackage{mathrsfs}
\usepackage{subfigure}
\usepackage{float}
\usepackage{color}
\usepackage{url,epsfig,tikz}
\usepackage{subfigure}
\usepackage{pgfplots}
\usepackage{epstopdf}
\usetikzlibrary{shapes.geometric}

\newcommand{\R}{{\mathbb R}}

\newcommand{\N}{{\mathbb N}}

\newcommand{\GG}{{\mathbb G}}
\newcommand{\I}{{\mathbb I}}

\newcommand{\be}{\begin{eqnarray}}
\newcommand{\ben}{\begin{eqnarray*}}
\newcommand{\en}{\end{eqnarray}}
\newcommand{\enn}{\end{eqnarray*}}

\newcommand{\pa}{\partial}

\newcommand{\real}{{\rm Re\,}}

\newcommand{\s}{\mathbb{S}}

\newcommand{\om}{\omega}

\newcommand{\rb}{\right)}
\newcommand{\lb}{\left(}

\newcommand{\supp}{{\rm supp}}

\newcommand{\bx}{\boldsymbol{x}}
\newcommand{\by}{\boldsymbol{y}}
\newcommand{\boe}{\boldsymbol{E}}
\newcommand{\bg}{\boldsymbol{g}}
\newcommand{\bof}{\boldsymbol{F}}
\newcommand{\bh}{\boldsymbol{H}}
\newcommand{\boh}{\boldsymbol{h}}
\newcommand{\bj}{\boldsymbol{J}}
\newcommand{\bp}{\boldsymbol{p}}
\newcommand{\bee}{\boldsymbol{e}}

\newcommand{\ka}{\kappa}

\newtheorem{thm}{Theorem}[section]

\newtheorem{lem}{Lemma}[section]

\newtheorem{defn}{Definition}[section]

\newtheorem{rem}{Remark}[section]

\definecolor{rot}{rgb}{1,0,0}
\definecolor{hw}{rgb}{0,0,1}

\begin{document}
\renewcommand{\theequation}{\arabic{section}.\arabic{equation}}
\begin{titlepage}
\title{\bf A direct sampling method for inverse time-dependent electromagnetic source problems: 
reconstruction of the radiating time and spatial support}
\author{Fengling Sun
 \thanks{School of Mathematical Sciences, Tianjin Normal University, 300387 Tianjin, P. R. China, ({\tt fenglinsun@tjnu.edu.cn})}
 \and Hongxia Guo \thanks{ (Corresponding author) School of Mathematical Sciences,  and Institute of Mathematics and Interdisciplinary
 Sciences, Tianjin Normal University, 300387 Tianjin, P. R. China. ({\tt hxguo@tjnu.edu.cn })}
}
\date{}
\end{titlepage}
\maketitle


\begin{abstract}

This paper investigates inverse source problems for time-dependent electromagnetic waves governed by Maxwell's equations. After applying the Fourier transform with respect to time, the problem leads to a frequency-domain electromagnetic system with a frequency-dependent source term.
We propose a novel direct sampling method for reconstructing such radiating time and spatial spaces of sources from multi-frequency far-field measurements. By using a pair of multi-frequency data from opposite observation directions, we can obtain the radiating time  of the signal. Based on this, the smallest region between two hyperplanes containing the support of source can be reconstructed using multi frequency data from one observation direction.  The $\Theta$ convex hull of the source support can be reconstructed from multi-frequency data from sparse observation directions. Compared with existing sampling methods that mainly focus on reconstructing
the spatial support, the proposed approach allows for the simultaneous
reconstruction of both spatial and temporal features of the source.Three-dimensional numerical examples are  conducted to validate the effectiveness of the algorithm.

\vspace{.2in} {\bf Keywords: Electromagnetic inverse source problem, direct sampling method, multi-frequency data.
}
\end{abstract}

\section{Introduction}
\subsection{Mathematical formulations}

Consider the time-dependent inverse source problem in electrodynamics, where the source depends on both time and spatial variables. The electromagnetic radiated wave propagation in an isotropic homogeneous dielectric medium in $\R^3$ with constant electric permittivity $\epsilon$ and magnetic permeability $\mu$, described by the electric field ${\boe}$ and magnetic field $\bh$ satisfying the Maxwell equations
\begin{equation} \label{eq:maxwell00}
	\begin{split}
    &\mu \pa_{t}{\bh(\bx,t)}+ \nabla \times \boe(\bx,t) =0, \\
    &\epsilon \pa_{t}{\boe(\bx,t)}-\nabla \times \bh(\bx,t) =-\bj(\bx,t),\quad (\bx,t) \in \R^3\times \R_+,
    \end{split}
\end{equation}
which is supplemented by the homogeneous initial conditions: 
\begin{equation*}
\begin{split}
&\boe(\bx,0)=\pa_t \boe(\bx,0) =0,\\
&\bh(\bx,0)=\pa_t \bh(\bx,0) =0,\quad \bx\in \R^3,
\end{split}
\end{equation*}
where $\bj$ represents the electric current density.
Let $D\subset \R^3$ is a bounded Lipschitz domain such that $\R^3\backslash\overline D$ is connected. We assume that $\bj\in (L^{\infty}((B_R)^3))^3$ was compactly supported within the domain $D$ enclosed in the ball $B_R=\{\bx\in\mathbb{R}^3:|\bx|<R \}$ for some $R>0$. Further assume that the compact support of the source term remains unchanged as time varies. 

In this paper, we consider the following two types of source:

(i)The source is assumed to emit a radiation wave signal at an unknown time $t_0$, which has the following expression with Dirac delta function $\delta$:
\begin{equation*}
    \bj(\bx,t)=\bj(\bx)\delta(t-t_0).
\end{equation*}

(ii) The electric current density is supposed to radiate over a finite time period $[t_{\min},t_{\max}]$ with $t_{\max}>t_{\min}\geq0$, where $t_{\min}$ denotes the initiation time of the electromagnetic wave signal and $t_{\max}$ denotes the termination time of the signal. Further assume that $\bj(\bx, t)\in L^\infty(L^2(B_R)^3\times (t_{\min}, t_{\max}))^3$ satisfies $\supp \bj (\bx, t)=D\subset B_R$ for all $t\in(t_{\min}, t_{\max})$.

Taking the inverse Fourier transform on the wave equation (\ref{eq:maxwell00}) with respect to the time  variable  yields the inhomogeneous Maxwell equation
\begin{equation} \label{eq:maxwell01}
	\begin{split}
		&\mathrm{i}\omega\mu {\widehat{\bh}(\bx,\omega)}-\nabla \times {\widehat{\boe}(\bx,\omega)} =0, \\
		&\mathrm{i}\omega\epsilon {\widehat{\boe}(\bx,\omega)} +\nabla \times {\widehat{\bh}(\bx,\omega)}=-\widehat{\bj}(\bx,\omega),
	\end{split}
\end{equation}
where $\omega$ is angle frequency, $\widehat{\boe}$, $\widehat{\bh}$ and $\widehat{\bj}$  denote the inverse Fourier transforms of the  $\boe$,  $\bh$ and $\bj$ with respect to the time $t$, respectively.
Eliminating the magnetic field $\widehat{\bh}$ from (\ref{eq:maxwell01}), we obtain the Maxwell system for the electric field $\widehat{\boe}$:
\begin{equation*} \label{eq:maxwell02}
 \nabla \times \nabla \times {\widehat{\boe}}(\bx,\omega)-\omega^2\epsilon\mu {\widehat{\boe}}(\bx,\omega)=\mathrm{i}\omega\mu\widehat{\bj}(\bx,\omega):=\mathrm{i}\omega\mu\bof(\bx,\omega)
 \quad (\bx,t) \in \R^3\times \R_+,
\end{equation*}
Furthermore, the scattered fields $\widehat{\boe}$ have to satisfy the Silver-M\"{u}ller radiation condition
\begin{equation*}
	\lim_{|\bx|\to\infty}((\nabla\times\widehat{\boe})\times {\bx}-\mathrm{i}\kappa|\bx|\widehat{\boe})=0,
\end{equation*}
where $\kappa=\omega\sqrt{\epsilon\mu}>0$ is the wavenumber.

Denote by $\GG$ the dyadic Green's tensor associated with Maxwell equations, satisfying 
\begin{equation*}
	\nabla\times\nabla\times\GG(\bx,\by)-\kappa^2\GG(\bx,\by)=\delta(\bx-\by)\mathbb{I}.	
\end{equation*}
where $\I$ is the matrix of identification $3\times3$. Explicitly, the dyadic Green's function $\GG$ can be expressed as
\begin{equation*}
	\GG(\bx,\by)=g(\bx,\by)\mathbb{I}+\dfrac{1}{\kappa^{2}}\nabla\nabla^{\mathsf T}g(\bx,\by),
\end{equation*}
$\nabla \nabla ^\mathsf{T} g(\bx,\by)$ is the Hessian matrix for $g(\bx,\by)$ defined by
\begin{equation*}
\lb\nabla \nabla^\mathsf{T} g(\bx,\by) \rb_{m,n}=\dfrac{\pa^2 g}{\pa x_m \pa x_n},\quad 1\leq m,n\leq 3.
\end{equation*}
We recall $g$ is the fundamental solution of the scalar Helmholtz equation in $\mathbb{R}^3$ and is given by 
\begin{equation*}
	g(\bx,\by)= \dfrac{\mathrm{e}^{\mathrm{i}\kappa|\bx-\by|}}{4\pi|\bx-\by|},\quad \bx\neq \by.
\end{equation*}
Then the scattered field $\widehat{\boe}$ has the representation
\begin{equation*}
	\widehat{\boe}(\bx,\omega)=\mathrm{i}\omega\mu\int_{\mathbb{R}^3}\GG(\bx,\by)\bof(\by,\omega)\mathrm{d}\by,\quad \bx\in\mathbb{R}^3.
\end{equation*}
It is well known that the Silver-M\"{u}ller radiation condition gives rise to the following asymptotic behavior of $\widehat{\boe}$ at infinity:
\begin{equation*}
    \widehat{\boe}(\bx,\omega)=\frac{\mathrm{e}^{\mathrm{i}\omega\sqrt{\epsilon\mu}|\bx|}}{4\pi|\bx|}\{\boe^\infty(\hat \bx,\om) +O(\frac{1}{|\bx|})\},\quad |\bx|\rightarrow \infty,
\end{equation*}
where $\boe^\infty(\hat \bx,\om)$ is known as the electric far field pattern of $\widehat{\boe}$, which is an analytic function on the unit sphere $\mathbb S^2$ with respect to the observation direction $\hat \bx\in \mathbb S^2$. Straightforward calculations shows that the corresponding electric far-field pattern $\boe^\infty(\hat \bx,\om)$ of $\widehat{\boe}$ is given by
\begin{equation}\label{e-infty}
    \boe^\infty(\hat \bx,\om) =\mathrm{i}\om\mu(\mathbb I -\hat\bx\hat\bx^\mathsf{T}) \int_{\R^3} \mathrm{e}^{-\mathrm{i}\om \sqrt{\epsilon\mu}\hat{\bx}\cdot \by } \bof(\by,\omega)\,\mathrm{d}\by,\quad \hat{\bx}\in \s^2.
\end{equation}

Based on different assumptions concerning the source terms, we consider the following two inverse source problems with far-field data measured in sparse observation directions for some $L\in \N$: 
$$\{\boe^\infty(\pm\hat{\bx}_l, \om): \om\in[\om_{\min}, \om_{\max}], l=1,2,\cdots, L\}.$$ 

(i)IP1: The inverse problem is to determine the unknown excitation time $t_0$ of the signal, then reconstruct the position and shape of the support $D$.

(ii)IP2: This inverse problem is to determine the unknown initiation time or termination time of the process, then reconstruct the position and shape of the support $D$.

The aim of this paper is to explore a direct sampling method to tackle the two inverse source problems from a numerical point of view.

\subsection{ Literature review}

This paper investigates the inverse source problem associated with the Maxwell equations. 
Due to its fundamental importance in applications such as antenna synthesis \cite{Ammari2002} 
and biomedical imaging \cite{Balanis2005,Dassios2005}, this problem has attracted sustained attention 
from engineers, physicists and mathematicians. 
However, the inverse source problem at a fixed frequency is severely ill-posed and does not admit a unique solution because of the existence of non-radiating sources 
\cite{Bleistein1977,Albanese2006,Valdivia2012,Li2024}. 
To overcome this difficulty, two main strategies have been widely adopted: incorporating a priori information about the source or exploiting multi-frequency data 
\cite{BaoAmmari2002,Bao2010,BaoLi2015,Liu2015,CIL}. 
In particular, the authors of \cite{Bao2020} established, for the first time, a unified increasing-stability estimate for inverse source problems governed by both the Navier and Maxwell equations using multi-frequency data. 
Further uniqueness and stability results can be found in 
\cite{LiYamamoto2005,HLLZ,Hukian2019,HL2020,Isakov2021,Yuan2023} and the references therein.

When the source term is independent of the wavenumber or frequency, the far-field pattern coincides with the Fourier transform of the spatial source distribution with the Fourier variable $\xi=k\hat{x}$. 
This property enables the development of various numerical reconstruction methods. 
Representative examples include recursive algorithms \cite{BaoLu2015}, Fourier-based reconstruction methods \cite{Wang2018,Wang2019}, and sampling-type techniques 
\cite{GS,Griesmaier2018,JL2019,LiLiu2023,Harris2024}. 
These approaches can generally be classified into iterative and non-iterative methods. 
Iterative methods are capable of recovering quantitative information about the source but usually require repeatedly solving the forward problem and depend on the choice of initial guesses, which results in considerable computational cost. 
To alleviate this difficulty, various non-iterative methods have been proposed. 
Since their indicator functions typically involve only inner products between measured data and suitably chosen test functions, these methods avoid repeated forward simulations and do not require prior information about the target, while maintaining high computational efficiency and good robustness. 
In general, sampling-type methods are qualitative in the sense that they primarily recover the support of the source rather than its amplitude. 
Recently, it has been shown in \cite{Liu2024,Liu2025} that sampling methods can even recover the source function itself provided that sufficiently many observation directions are available, leading to significantly improved reconstruction resolution.

In many practical situations, the source depends on the wavenumber or frequency. 
Such wavenumber-dependent sources naturally arise when applying the Fourier transform to time-dependent source terms. 
In this case, however, the far-field pattern no longer coincides with the Fourier transform of the source function, which significantly complicates the theoretical analysis and the design of reconstruction algorithms. 
For the acoustic case, the authors of \cite{AHLS} showed that multi-frequency data measured from a single observation direction or point can provide useful information about the position and size of the support of the source term. 
Subsequently, factorization-type approaches were extended to wavenumber-dependent source problems with known radiation periods, providing necessary and sufficient conditions for imaging the slab that contains the support of the source term \cite{GGH}. 
To improve practical implementability, a direct sampling method under weaker mathematical assumptions was further proposed in \cite{GHZ}. 
Moreover, uniqueness and increasing-stability results for acoustic inverse source problems with unknown wavenumber dependence were established in \cite{Zhao2024}, together with two reconstruction algorithms. 
More recently, \cite{Ma2025} employed multi-frequency far-field data from two opposite observation directions and applied the factorization method to determine the pulse moment of a moving extended source.

Despite these developments, the study of wavenumber-dependent inverse source problems in the electromagnetic setting remains relatively limited. 
Extending acoustic reconstruction techniques to Maxwell equations is generally non-trivial due to the vectorial nature of electromagnetic fields and the presence of polarization effects. 
Motivated by the ideas in \cite{JL2019,Ma2025}, we preprocess the electromagnetic far-field data by selecting suitable polarization directions, which allows the vectorial Maxwell system to be reduced to a form that can be handled using direct sampling techniques. 
From a physical perspective, this strategy also provides a natural way to analyze time-varying electromagnetic inverse source problems in the frequency domain.

The main contributions of this work are summarized as follows. 
To the best of our knowledge, this is the first direct sampling framework for wavenumber-dependent electromagnetic inverse source problems using  multi-frequency far-field measurements. 
First, we extend the direct sampling framework developed for wavenumber-dependent acoustic inverse source problems to the electromagnetic case governed by the Maxwell equations. 
Second, by utilizing multi-frequency far-field data measured from a pair of symmetric observation directions, we construct indicator functions that allow the identification of the excitation starting and starting/terminal times of the signal. 
Third, based on these temporal reconstructions, we determine the smallest slab region bounded by two hyperplanes that contains the support of the source term. 
Compared with the work published in \cite{GHZ}, where the reconstruction primarily focused on spatial information under a different measurement configuration, the present work employs symmetric multi-frequency observations and provides a unified framework for simultaneously recovering both temporal information and the spatial support of the source.

The remainder of this paper is structured as follows. In Section 2, we first present the inverse Fourier transform of far-field data and then concentrate on the selection of test and indicator functions using multi-frequency far-field data. Additionally, in Section 3, we will extend the direct sampling method to the situations in which the source term radiates signals over a finite time period. Numerical tests will be implemented in Section 4. Finally, some concluding remarks are given in Section 5.


Below we introduce several notations that will be used throughout the paper.

Let $\hat{\bx}\in \mathbb{S}^2$ be an observation direction and 
$D\subset \mathbb{R}^3$ be a bounded domain with a $C^2$ boundary. 
Define
\begin{equation}\label{hatxcdotD}
\hat{\bx}\cdot D
=
\{t\in \mathbb{R}: t=\hat{\bx}\cdot \by \ \text{for some}\ \by\in D\}.
\end{equation}
Obviously, $\hat{\bx}\cdot D$ is an interval of $\mathbb{R}$. 
Moreover, it holds that
\begin{equation}\label{relation}
\inf(-\hat\bx\cdot D)=-\sup(\hat\bx\cdot D),\qquad 
\sup(-\hat\bx\cdot D)=-\inf(\hat\bx\cdot D).
\end{equation}

Throughout the paper, the one-dimensional Fourier transform and its inverse are defined by
\begin{equation*}
\begin{split}
(\mathcal{F}g)(\alpha)
&=
\frac{1}{\sqrt{2\pi}}
\int_{\mathbb{R}} g(\omega)\mathrm{e}^{-\mathrm{i}\omega\alpha}\,\mathrm{d}\omega,
\quad \alpha\in \mathbb{R},\\
(\mathcal{F}^{-1}f)(\omega)
&=
\frac{1}{\sqrt{2\pi}}
\int_{\mathbb{R}} f(\alpha)\mathrm{e}^{\mathrm{i}\omega\alpha}\,\mathrm{d}\alpha,
\quad \omega\in \mathbb{R}.
\end{split}
\end{equation*}

\begin{defn}
Let $\bj\in L^2(\mathbb{R})^3$. 
The \emph{supporting interval} of $\bj$ is defined as the minimal interval 
$I\subset\mathbb{R}$ such that $\bj$ vanishes almost everywhere outside $I$.
\end{defn}

\section{Inverse source problem 1}\label{sec:3}
\subsection{Inverse Fourier transform of far-field pattern}\label{sec:3.1}

In this section, we suppose that the source emits a signal at unknown time $t_0$. Then $\bof(\bx,\omega)$ can be expressed as
\begin{equation}\label{source1}
\bof(\bx,\omega)=\bj(\bx)\mathrm{e}^{\mathrm{i}\omega t_0}.
\end{equation}

\begin{rem}
If $t_0=0$, then the source term becomes
\[
\bof(\bx,\omega)=\bj(\bx),
\]
which is independent of the frequency $\omega$. 
Therefore, the considered model reduces to the classical inverse source problem with frequency-independent sources. 
Hence the present formulation can be regarded as a natural extension of the standard frequency-independent inverse source problem that has been widely studied in the literature.
\end{rem}

Firstly, for any fixed $\hat\bx\in\mathbb{S}^2$, we preprocess the electric far-field field data and denote the preprocessed data as the blackboard bold $\mathbb{E}^{\infty}(\hat \bx,\om)$, or explicitly
\begin{equation*}
    \mathbb{E}^{\infty}(\hat \bx,\om):=\dfrac{1}{\mathrm{i}\om\mu}\bp^{(\hat\bx)}\cdot\boe^{\infty}(\hat \bx,\om),
\end{equation*}
where $\bp^{(\hat\bx)}\in\mathbb{S}^2$  denotes an arbitrary vector orthogonal to the observation direction $\hat\bx$. Below we describe the supporting interval of the inverse Fourier transform of the preprocessed far-field pattern with respect to frequency at a fixed observation direction. Moreover, for simplicity, we denote the wave speed $c=\dfrac{1}{\sqrt{\epsilon\mu}}$.

\begin{lem}\label{21}
	For any fixed $\hat\bx\in\mathbb{S}^2$. The supporting interval of $\mathcal{F}^{-1}(\mathbb{E}^{\infty}(\hat \bx,\om))(t)$ is $H:=(c^{-1}\inf(\hat\bx\cdot D)-t_0,c^{-1}\sup(\hat\bx\cdot D)-t_0)$.
\end{lem}
\begin{proof}
    According to \eqref{e-infty} and \eqref{source1}, for any $\hat\bx\in\mathbb{S}^2$, the inner product of $\bp^{(\hat\bx)}$ and the electric far-field pattern $\boe^{\infty}(\hat \bx,\om)$ can be expressed as
    \begin{align*}
        \bp^{(\hat\bx)}\cdot\boe^{\infty}(\hat \bx,\om)
        & = \bp^{(\hat\bx)} \cdot \mathrm{i}\om\mu(\mathbb I -\hat\bx\hat\bx^\mathsf{T}) \int_{\R^3} \mathrm{e}^{-\mathrm{i}\om c^{-1}(\hat{\bx}\cdot \by -ct_0)} \bj(\by)\,\mathrm{d}\by \\
        & = \mathrm{i}\om\mu \int_{\R^3} \mathrm{e}^{-\mathrm{i}\om c^{-1}(\hat{\bx}\cdot \by -ct_0)} \bp^{(\hat\bx)}\cdot \bj(\by)\,\mathrm{d}\by.
    \end{align*}

To write the above expression as a Fourier transform, we observe that
\begin{equation*}
    \int_{\R^3} \mathrm{e}^{-\mathrm{i}\om c^{-1}(\hat{\bx}\cdot \by -ct_0)} \bp^{(\hat\bx)}\cdot\bj(\by)\,\mathrm{d}\by
    =\int_{\R} \mathrm{e}^{-\mathrm{i}\om\alpha}\int_{\Gamma(c(t_0+\alpha))} \bp^{(\hat\bx)}\cdot \bj(\by)\mathrm{d}s(\by)\mathrm{d}\alpha
\end{equation*}
where $\Gamma(t_0+\alpha)\subset \mathbb{R}$ is defined as
    \be
\Gamma(c(t_0+\alpha)):=\left\lbrace \by\in D:\hat{\bx}\cdot \by=c(t_0+\alpha) \right\rbrace , \quad t_0\in \R. \nonumber
    \en
Note that when $\Gamma(c(t_0+\alpha))=\emptyset$, the aforementioned integral over $\Gamma(c(t_0+\alpha))$ is taken as zero. Thus, $\mathbb{E}^{\infty}(\hat \bx,\om)$ admits the following expression 
\begin{equation}\label{eq:Fouriertransform}
     \mathbb{E}^{\infty}(\hat \bx,\om)
     = \int_{\R} \mathrm{e}^{-\mathrm{i}\om\alpha}\bg(\alpha)\mathrm{d}\alpha
    = (\mathcal{F}\bg)(\om),
\end{equation}
with
\begin{equation}\label{eq:g}
   \bg(\alpha)=\int_{\Gamma(c(t_0+\alpha))} \bp^{(\hat\bx)}\cdot \bj(\by)\mathrm{d}s(\by).
\end{equation}
 
 Since $\Gamma(c(t_0+\alpha))=\emptyset$ for 
    $c(t_0+\alpha)<\inf(\hat\bx\cdot D)$ or
    $c(t_0+\alpha)>\sup(\hat\bx\cdot D)$, it is obvious that 
    \begin{equation*}
        \bg(\alpha)=0 \qquad\qquad\mbox{if} \quad \alpha<c^{-1}\inf(\hat\bx\cdot D)-t_0 \quad\mbox{or}\quad \alpha>c^{-1}\sup(\hat\bx\cdot D)-t_0,
    \end{equation*}
which implies 
    \begin{equation*}\label{suppg}
    \mathrm{supp}\bg(\alpha)\subset(c^{-1}\inf(\hat\bx\cdot D)-t_0,c^{-1}\sup(\hat\bx\cdot D)-t_0).
    \end{equation*}
Let $H:=(c^{-1}\inf(\hat\bx\cdot D)-t_0,c^{-1}\sup(\hat\bx\cdot D)-t_0)$, we deduce that the supporting interval of $\mathcal{F}^{-1}(\mathbb{E}^{\infty}(\hat \bx,\om))(t)$ is $H$. 
\end{proof}

\begin{defn}
    For any fixed unit vector $\hat\bx\in\mathbb{S}^2$, we define the slab $S(\hat\bx;\bj,t_0)$ as the region between two parallel hyperplanes orthogonal to $\hat\bx$, given by
    \begin{equation}
        S(\hat\bx;\bj,t_0):=\{\by\in\mathbb{R}^3:c^{-1}\inf{(\hat\bx\cdot D)}-t_0<\hat\bx\cdot\by<c^{-1}\sup{(\hat\bx\cdot D)}-t_0\}.
    \end{equation}
\end{defn}
    
Building on the preceding analysis, we establish the following uniqueness theorem for the slab.

\begin{thm}\label{thm1}
    For any fixed $\hat\bx\in\mathbb{S}^2$, choose a vector $\bp^{(\hat\bx)}\in\mathbb{S}^2$ such that $\bp^{(\hat\bx)}\cdot\hat\bx=0$. If the set
    \begin{equation}\label{zeroset1}
     \{\alpha\in\mathbb{R}:\Gamma(c(t_0+\alpha))\subset S(\hat\bx;\bj,t_0),\bg(\alpha)=0\}
    \end{equation}
    has Lebesgue measure zero, then the slab $S(\hat\bx;\bj,t_0)$ can be uniquely determined by the data  $\mathbb{E}^{\infty}(\hat \bx,\om)$ for all $\om\in(\om_{\min},\om_{\max})$ at the observation direction $\hat\bx\in\mathbb{S}^2$.
\end{thm}
\begin{proof}
    From the representation \eqref{eq:Fouriertransform}, we deduce that the the preprocessed data $\mathbb{E}^{\infty}(\hat \bx,\om)$ is just the Fourier transform of $\bg(\alpha)$. 
    This implies that $\bg(\alpha)$ can be uniquely determined by $\mathbb{E}^{\infty}(\hat \bx,\om)$, where $\om\in(\om_{\min},\om_{\max})$. Note that the electric far field pattern is analytic in angle frequency $\om$. Under the assumption that the Lebesgue measure of  \eqref{zeroset1} is zero, we can conclude that 
    \begin{equation*}
    S(\hat\bx;\bj,t_0)=\overline{\bigcup_{\alpha\in\mathbb{R}}\{\Gamma(c(t_0+\alpha))|\bg(\alpha)\neq0}\}.
    \end{equation*}
    which implies that the slab $S(\hat\bx;\bj,t_0)$ is uniquely determined by $\bg(\alpha)$, also by the data $\mathbb{E}^{\infty}(\hat \bx,\om)$. 
\end{proof}

\begin{rem}
   The set in \eqref{zeroset1} has Lebesgue measure zero if the real part of a complex multiple of the source projection function $\bp^{(\hat\bx)}\cdot \bj(\by)$ is bounded away from zero on their support, i.e., we assume that for some $\alpha\in\mathbb{R}$ and $c>0$, $\bp^{(\hat\bx)}\cdot \bj(\by) \in L^{\infty}(D)$ satisfies 
   \begin{equation}\label{positivity condition}
       \real(\mathrm{e}^{\mathrm{i}\alpha}\bp^{(\hat\bx)}\cdot \bj(\by)) \geq c_0\quad a.e.  \quad\by\in D.
   \end{equation}
\end{rem}

\subsection{Indicator and test functions}\label{sec:3.2}

From Theorem \ref{thm1}, one can extract information on the interval $(c\inf(\hat\bx\cdot D)-t_0,c\sup(\hat\bx\cdot D)-t_0)$ by taking the inverse Fourier transform of the data $\mathbb{E}^{\infty}(\hat \bx,\om)$. In this section, we first design a proper indicator function and employ the far-field data collected from two opposite observation directions to determine the excitation time $t_0$. 

For any sampling point $\by\in\mathbb{R}^3$ and a given time parameter 
$\eta\ge0$, we define the test function $\phi^{(\hat{\bx})}_{\eta}(\by,\omega)$ by
\begin{equation*}
\phi^{(\hat{\bx})}_{\eta}(\by,\omega)
:=
\mathrm{e}^{-\mathrm{i}\omega(c^{-1}\hat{\bx}\cdot \by-\eta)} .
\end{equation*}
Here $\by$ represents the spatial sampling point, while $\eta$ denotes a time sampling parameter. 
In the proposed sampling scheme, the parameter $\by$ is used to probe the spatial location of the source, whereas $\eta$ serves to extract temporal information of the emitted signal.
Then we introduce the following auxiliary indicator function
\begin{equation}\label{Indicator}
    \mathcal I^{(\hat{\bx})}_{\eta}(\by) 
    =\int_{\R}\mathbb{E}^{\infty}(\hat \bx,\om)
    \overline{\phi_{\eta}^{(\hat{\bx})}(\by,\omega)} \, d{\omega}.
\end{equation}

To characterize the spatial behavior of the indicator function, 
we introduce several slab regions associated with the projection of the domain $D$ along the observation direction $\hat{\bx}$. 
Recall the definition of $\hat{\bx}\cdot D$ in \eqref{hatxcdotD} and the relation in \eqref{relation}. 
Based on these quantities, we define the following unbounded parallel slabs
\begin{equation}\label{stripK_D}
    \begin{split}
        &K_{D}^{(\hat\bx)}
        :=\{\by\in\mathbb{R}^3: 
        \inf(\hat\bx\cdot D)<\hat\bx\cdot\by< \sup(\hat\bx\cdot D)\},\\
        &K_{D,\eta}^{(\hat\bx)}
        :=\{\by\in\mathbb{R}^3: 
        \inf(\hat\bx\cdot D)-c(t_0-\eta)
        <\hat\bx\cdot\by<
        \sup(\hat\bx\cdot D)-c(t_0-\eta)\},\\
        &K_{D,\eta}^{(-\hat\bx)}
        :=\{\by\in\mathbb{R}^3: 
        \inf(\hat\bx\cdot D)+c(t_0-\eta)
        <\hat\bx\cdot\by<
        \sup(\hat\bx\cdot D)+c(t_0-\eta)\}.
    \end{split}    
\end{equation}
These slabs are perpendicular to the observation direction $\hat\bx$. 
Obviously, $K_{D}^{(\hat\bx)}$ is the smallest slab containing $\overline{D}$. 
Moreover, the slabs $K_{D,\eta}^{(\hat\bx)}$ and $K_{D,\eta}^{(-\hat\bx)}$ can be viewed as translations of $K_{D}^{(\hat\bx)}$ along or opposite to the direction $\hat\bx$ by the distance $c(t_0-\eta)$, respectively.

The following theorem reveals the fundamental property of the indicator function and explains how the support information of the source is encoded in the slab $K_{D,\eta}^{(\hat\bx)}$.

\begin{thm}\label{Th3-2}
    Let $\hat{\bx}\in\mathbb{S}^2$ be fixed. Then the indicator function satisfies
    \begin{equation}\label{IndicatorI1}
        \mathcal I^{(\hat{\bx})}_{\eta}(\by)
        =\left
        \{\begin{array}{lll}
        \mbox{a finite positive number}, 
        &&\mbox{for} \quad \by\in K_{D,\eta}^{(\hat\bx)},\\[4pt]
        0,
        &&\mbox{for}\quad \by\notin \overline{K_{D,\eta}^{(\hat\bx)}}.
       \end{array}\right.
    \end{equation}
\end{thm}

\begin{proof}
    The test function $\phi^{(\hat{\bx})}_{\eta}(\by,\om)$ can be rephrased as 
    \begin{equation*}\label{Fourierdelta}
        \phi^{(\hat{\bx})}_{\eta}(\by,\om):=\int_{\R} \mathrm{e}^{-\mathrm{i}\om\alpha}\delta[\alpha-(c^{-1}\hat\bx\cdot\by-\eta)] \mathrm{d}\alpha
        = [\mathcal{F}(\delta(\alpha-c^{-1}\hat\bx\cdot\by+\eta))](\om).
    \end{equation*}
Applying the Parserval’s identity and properties of Fourier transform, we deduce from \eqref{eq:Fouriertransform} and \eqref{Fourierdelta} that
\begin{align*}
    \mathcal I^{(\hat{\bx})}_{\eta}(\by,\om) &=\int_{\R}\mathbb{E}^{\infty}(\hat \bx,\om)\overline{\phi^{(\hat{\bx})}_{\eta}(\by,\om)} \mathrm{d}{\om} \\
    &=\int_{\R}[\mathcal{F}\bg(\alpha)](\om)\overline{[\mathcal{F}(\delta(\alpha-c^{-1}\hat\bx\cdot\by+\eta))](\om)}\mathrm{d}\omega \\
    &=\int_{\R}\bg(\alpha)\delta(\alpha-c^{-1}\hat\bx\cdot\by+\eta)\mathrm{d}\alpha \\
    &=\bg(c^{-1}\hat\bx\cdot\by-\eta),
\end{align*}
where $\bg(\alpha)$ is defined by \eqref{eq:g}. From \eqref{suppg}, we know that $\mathrm{supp}\bg(\alpha)\subset(c^{-1}\inf(\hat\bx\cdot D)-t_0,c^{-1}\sup(\hat\bx\cdot D)-t_0)$. Therefore $\bg(c^{-1}\hat\bx\cdot\by-\eta)\neq 0$ for all $\by\in\mathbb{R}^3$ such that $c^{-1}\hat\bx\cdot\by-\eta\in(c^{-1}\inf(\hat\bx\cdot D)-t_0,c^{-1}\sup(\hat\bx\cdot D)-t_0)$, that is $\by\in K_{D,\eta}^{(\hat\bx)}$. On the other hand, it is also obvious that $\bg(c^{-1}\hat\bx\cdot\by-\eta)= 0$ for all $\by\notin K_{D,\eta}^{(\hat\bx)}$. 
which proves \eqref{IndicatorI1}.
\end{proof}

Now we turn to the recovery of the unknown excitation time $t_0$. 
To this end, we combine the indicator functions corresponding to a pair of opposite observation directions. 
Specifically, we introduce the following auxiliary indicator function
\begin{equation*}
    \mathcal W_{\eta}^{(\hat{\bx})}(\by):=
    \left[
    \frac{1}{\mathcal I_{\eta}^{(\hat{\bx})}(\by)}+
    \frac{1}{\mathcal I_{\eta}^{(-\hat{\bx})}(\by)}
    \right]^{-1}
    =
    \frac{\mathcal I_{\eta}^{(\hat{\bx})}(\by)\; \mathcal I_{\eta}^{(-\hat{\bx})}(\by)}
    {\mathcal I_{\eta}^{(\hat{\bx})}(\by)+\mathcal I_{\eta}^{(-\hat{\bx})}(\by)}.
\end{equation*}

This construction allows us to extract the common spatial region determined by the two opposite observation directions. 
As a consequence, the following result follows directly from Theorem~\ref{Th3-2} together with the definition of the slabs in \eqref{stripK_D}.

\begin{thm}\label{IndicatorW}
    Let $\hat{\bx}\in\mathbb{S}^2$ be fixed. Then
    \begin{equation}\label{IndicatorW1}
        \mathcal W_{\eta}^{(\hat{\bx})}(\by)=\left
        \{\begin{array}{lll}
        \mbox{a finite positive number}, 
        &&\mbox{for} \quad \by\in K_{D,\eta}^{(\hat\bx)}\cap K_{D,\eta}^{(-\hat\bx)},\\[4pt]
        0, 
        &&\mbox{for}\quad \by\notin \overline{K_{D,\eta}^{(\hat\bx)}\cap K_{D,\eta}^{(-\hat\bx)}}.
       \end{array}\right.
    \end{equation}
\end{thm}

\begin{proof}
    From \eqref{IndicatorI1} and \eqref{IndicatorW1}, 
    it is obvious that for $\by\in K_{D,\eta}^{(\hat\bx)}\cap K_{D,\eta}^{(-\hat\bx)}$, both $\mathcal I^{(\hat{\bx})}_{\eta}(\by,\om)$ and $\mathcal I^{(-\hat{\bx})}_{\eta}(\by,\om)$ are finite positive number. Implying that $0<\mathcal W_{\eta}^{(\hat{\bx})}(\by,\om) <\infty$.

    On the other hand, for $\by\notin \overline{K_{D,\eta}^{(\hat\bx)}\cap K_{D,\eta}^{(-\hat\bx)}}$, without loss of generality, one can assume that $\by\notin \overline{K_{D,\eta}^{(\hat\bx)}}$. In this case, $\mathcal I^{(\hat{\bx})}_{\eta}(\by,\om)=0$ and then $\mathcal W_{\eta}^{(\hat{\bx})}(\by,\om)=0$. 
\end{proof}

\begin{rem}
The construction of $\mathcal W_{\eta}^{(\hat{\bx})}$ acts as a filtering mechanism 
that extracts the intersection of the supports of the two indicator functions 
$\mathcal I_{\eta}^{(\hat{\bx})}$ and $\mathcal I_{\eta}^{(-\hat{\bx})}$. 
Consequently, the resulting indicator localizes the spatial region determined 
simultaneously by the two opposite observation directions.
\end{rem}

When $\eta$ deviates from the true excitation time, the two reconstructed slabs remain separated and the corresponding indicator function vanishes over most of the sampling region. As $\eta$ approaches the true excitation moment, the two slabs gradually overlap and the intersection region enlarges. The overlap reaches its maximal extent when $\eta$ is closest to the true excitation time.
This geometric transition provides a stable numerical mechanism for identifying the unknown excitation time. To determine $t_0$, we consider the one-dimensional function
\[
\eta \mapsto \max_{\by\in B_{\mathbb{R}}}\mathcal W_{\eta}^{(\hat{\bx})}(\by),
\]
whose behavior enables the recovery of the excitation time.

\begin{thm}\label{timeindicator}
Define the function
\[
\mathcal T_{t_0}^{(\hat{\bx})}(\eta)
=\max_{\by\in B_{\mathbb{R}}}\mathcal W_{\eta}^{(\hat{\bx})}(\by).
\]
Then the following property holds
\begin{equation}
\mathcal T_{t_0}^{(\hat{\bx})}(\eta)=
\left
\{
\begin{array}{lll}
\mbox{positive} && \mbox{for}\quad \eta\in [\eta_1,\eta_2],\\
0 && \mbox{for}\quad \eta\notin [\eta_1,\eta_2].
\end{array}
\right.
\end{equation}
Consequently, the unknown excitation time is given by
\[
t_0=\dfrac{\eta_1+\eta_2}{2}.
\]
\end{thm}

\begin{proof}
Combining the definition of the slabs in \eqref{stripK_D} with Theorem \ref{IndicatorW}, we observe that $\eta_1$ and $\eta_2$ correspond to the critical moments when the slabs $K_{D,\eta}^{(\hat\bx)}$ and $K_{D,\eta}^{(-\hat\bx)}$ start to overlap and become completely separated, respectively. A direct calculation yields
\begin{equation*}
\begin{split}
\eta_1 &= t_0-\dfrac{1}{2c}
\big[\sup(\hat\bx\cdot D)-\inf(\hat\bx\cdot D)\big],\\
\eta_2 &= t_0-\dfrac{1}{2c}
\big[\inf(\hat\bx\cdot D)-\sup(\hat\bx\cdot D)\big].
\end{split}
\end{equation*}
Taking the average of $\eta_1$ and $\eta_2$ immediately gives
\[
t_0=\dfrac{\eta_1+\eta_2}{2},
\]
which completes the proof.
\end{proof}

\begin{rem}
The relation $t_0=(\eta_1+\eta_2)/2$ admits a clear geometric interpretation. 
The slabs $K_{D,\eta}^{(\hat\bx)}$ and $K_{D,\eta}^{(-\hat\bx)}$ can be viewed as translations of the smallest slab $K_D^{(\hat\bx)}$ along the directions $\hat\bx$ and $-\hat\bx$, respectively, with a shift distance proportional to $c^{-1}(t_0-\eta)$. 
When $\eta$ varies, these two slabs move in opposite directions and their intersection changes accordingly. 
The parameters $\eta_1$ and $\eta_2$ correspond to the two critical moments when the slabs begin to intersect and finally separate. 
Due to the symmetric propagation of electromagnetic waves along opposite directions, these two moments are located symmetrically around the true excitation time $t_0$. 
Therefore, the midpoint of the interval $[\eta_1,\eta_2]$ naturally yields the excitation time.
\end{rem}

Once the excitation time $t_0$ has been determined according to Theorem \ref{timeindicator}, the smallest slab $K_{D}^{(\hat\bx)}$ containing the source support can be reconstructed by the following indicator function using the multi-frequency electric far-field pattern measured in a single observation direction
\begin{equation}\label{IndicatorIII4.2}
\mathcal I^{(\hat{\bx})}(\by)
=\int_{\R}
\mathbb{E}^{\infty}(\hat \bx,\omega)
\overline{\phi^{(\hat{\bx})}(\by,\omega)}
\,\mathrm{d}\omega,
\end{equation}
where
\[
\phi^{(\hat{\bx})}(\by,\omega)
:=\mathrm{e}^{-\mathrm{i}\omega(c^{-1}\hat\bx\cdot \by-t_0)},
\qquad \by\in\mathbb{R}^3.
\]

Clearly, the indicator function $\mathcal I^{(\hat{\bx})}(\by)$ satisfies
\begin{equation}\label{I_1}
\mathcal I^{(\hat{\bx})}(\by)=
\left
\{
\begin{array}{lll}
\mbox{a finite positive number}
&& \mbox{for}\quad \by\in K_{D}^{(\hat\bx)},\\
0
&& \mbox{for}\quad \by\notin \overline{K_{D}^{(\hat\bx)}}.
\end{array}
\right.
\end{equation}

When multiple but sparse observation directions $\{\hat\bx_l\}_{l=1}^{L}$ are available, we denote by $\mathcal I^{(\hat{\bx_l})}(\by)$ the indicator function corresponding to the $l$-th observation direction. 
To combine the information from different directions, we introduce the following composite indicator function
\begin{equation}\label{I_1^{-1}}
\mathcal I(\by)
=\left[
\sum_{l=1}^{L}
\dfrac{1}{\mathcal I^{(\hat{\bx}_l)}(\by)}
\right]^{-1}.
\end{equation}

When $\eta$ deviates from the true excitation moment,
the reconstructed slabs remain separated and the indicator function
vanishes in most of the sampling region.
As $\eta$ approaches the correct excitation time,
the intersection region gradually enlarges and reaches its maximal extent.
This transition provides a stable numerical mechanism
for identifying the unknown excitation time. By plotting the one-dimensional function $\eta\to\max\limits_{\by\in B_{\mathbb{R}}}\mathcal W_{\eta}^{(\hat{\bx})}(\by,\om)$, the unknown excitation time $t_0$ can be recovered.

\begin{thm}\label{timeindicator}
    The function $\mathcal T_{t_0}^{(\hat{\bx})}(\eta)=\max\limits_{\by\in B_{\mathbb{R}}}\mathcal W_{\eta}^{(\hat{\bx})}(\by,\om)$ fulfills that
    \begin{equation}
        \mathcal T_{t_0}^{(\hat{\bx})}(\eta)=\left
        \{\begin{array}{lll}
        \ge 0 \quad&&\mbox{for} \quad 
        \eta\in [\eta_1,\eta_2],\\
        0\quad&&\mbox{for}\quad \eta\notin [\eta_1,\eta_2].
       \end{array}\right.
    \end{equation}
    Then the unknown excitation time $t_0$ is
    \begin{equation*}
        t_0=\dfrac{\eta_1+\eta_2}{2}.
    \end{equation*}
\end{thm}
\begin{proof}
    Combining the definition of \eqref{stripK_D} and Theorem \ref{IndicatorW}, we know that $\eta_1$ and $\eta_2$ denote the moment when $K_{D,\eta}^{(\hat\bx)}$ and $K_{D,\eta}^{(-\hat\bx)}$ begin to converge and completely separate, respectively. Furthermore, we can infer that 
    \begin{equation*}
        \begin{split}
            &\eta_1=t_0-\dfrac{1}{2c}[\sup(\hat\bx\cdot D)-\inf(\hat\bx\cdot D)], \\
            &\eta_2=t_0-\dfrac{1}{2c}[\inf(\hat\bx\cdot D)-\sup(\hat\bx\cdot D)].
        \end{split}
    \end{equation*}
    By simple calculation, it can be concluded that $t_0=\dfrac{\eta_1+\eta_2}{2}$.
\end{proof}

After determining the source excitation time $t_0$ according to Theorem \ref{timeindicator}, the smallest slab $K_{D}^{(\hat\bx)}$ containing
the source support can be determined by the following indicator function using multi-frequency electricity far field pattern at a single direction,
\begin{equation*}\label{IndicatorIII4.2}
    \mathcal I^{(\hat{\bx})}(\by,\om) =\int_{\R}\mathbb{E}^{\infty}(\hat \bx,\om)\overline{\phi^{(\hat{\bx})}(\by,\om)} \mathrm{d}{\om}.
\end{equation*}
where $\phi^{(\hat{\bx})}(\by,\om):=\mathrm{e}^{-\mathrm{i}\om(c^{-1}\hat\bx\cdot \by-t_0)}, \by\in\mathbb{R}^3$. Clearly, the indicator function $\mathcal I^{(\hat{\bx})}(\by,\om)$ fulfills that
\begin{equation}\label{I_1}
    \mathcal I^{(\hat{\bx})}(\by,\om)=\left
    \{\begin{array}{lll}
    \mbox{finite positive number}\quad&&\mbox{for} \quad \by\in K_{D}^{(\hat\bx)},\\
     0\quad&&\mbox{for}\quad \by\notin \overline{K_{D}^{(\hat\bx)}}.
     \end{array}\right.
\end{equation}

With multiple but sparse observations directions ${\hat\bx_l}$, we use $\mathcal I^{(\hat{\bx_l})}(\by,\om)$ to denote the indicator function corresponding to each observation direction, where $l = 1,2,...,L$. 
Next, we construct a new indicator function by summing these functions and taking the reciprocal, namely
\begin{equation}\label{I_1^{-1}}
    \mathcal I (\by,\om) = \left[\sum_{l=1}^L \dfrac{1}{\mathcal I^{(\hat{\bx}_l)}(\by,\om)}\right]^{-1}.
\end{equation}
Recall that the convex hull of the source region $D$, defined as the intersection of all half‑spaces containing $D$, can be approximated by the intersection of the slabs $K_{D}^{(\hat\bx_l)}$. We therefore define the $\Theta$-convex hull of $D$ associated with the observation directions ${\hat\bx_l }$ as
\begin{equation*}
    \Theta_D:=\bigcap_{l=1}^{L}K_{D}^{(\hat\bx_l)}.
\end{equation*}
Here $K_{D}^{(\hat\bx_l)}$ denotes the slab determined by the electric far-field data measured in the observation directions $\hat\bx_l$. In this manner, the geometry of the $\Theta$-convex hull of the source support can be recovered.

\begin{thm}\label{convexhull-indicator}
    Let $D$ be connected. Then we have 
     \begin{equation*}
    \mathcal I(\by,\om)=\left
    \{\begin{array}{lll}
    \mbox{finite positive number}\quad&&\mbox{for} \quad \by\in \Theta _{D},\\
     0\quad&&\mbox{for}\quad \by\notin \overline\Theta _{D}.
     \end{array}\right.
\end{equation*}
\end{thm}

\begin{proof}
    Since $\Theta_D$ is the intersection of the slabs $K_{D}^{(\hat\bx_l)}$, it follows immediately that $\Theta_D\subset K_{D}^{(\hat\bx_l)}$ for all $l=1,2,\cdots,L$. 
    If $\by\in \Theta _{D}$, then $\by\in K_{D}^{(\hat\bx_l)}$, using the indicator function \eqref{I_1}, we conclude that $\mathcal I (\by,\om)$ is finite and positive. 
    On the other hand, if $\by\notin \Theta _{D}$, then there exists at least one $\hat\bx_l$ such that $\by\notin K_{D}^{(\hat\bx_l)}$ and $\dfrac{1}{\mathcal I^{(\hat{\bx_{l}})}(\by,\om)}=\infty$. It then follows that $\mathcal I (\by,\om)=0$. 
\end{proof}

With all observations $\hat\bx\in\mathbb{S}^2$, one can easily obtain the following uniqueness result.
\begin{thm}
    Assume the condition \eqref{positivity condition} holds on the connected domain $D\subset \mathbb{R}^3$. Then the convex hull of $D$ is uniquely determined by the multi-frequency far-field patterns $\{\boe^\infty(\hat{\bx}, \om): \hat\bx\in\mathbb{S}^2, \om\in[0, W]\}$ with some $W>0$. 
\end{thm}

\begin{rem}
    It should be noted that the information contained in $\boe^\infty(\hat{\bx}, \om)$ for $W>0$ is equivalent to the complete far-field data for all $\om>0$, due to the analytic dependence of $\boe^\infty(\hat{\bx}, \om)$ on $\om$.
\end{rem}

\section{Inverse source problem 2}\label{sec:4}
\subsection{Inverse Fourier transform of far-field pattern}\label{sec:4.1}

In this section, we assume that the source radiates over a finite time interval $[t_{\min},t_{\max}]$ , for which either the excitation time  or the termination time is known. Obviously, $\bof(\bx,\omega)$ can be expressed as
\begin{equation*}
 \bof(\bx,\omega)=\int_{t_{\min}}^{t_{\max}}\bj(\bx,t)\mathrm{e}^{\mathrm{i}\omega t}\mathrm{d}t.
\end{equation*}
It is clear that the time‑dependent source function considered in the previous section can be expressed as a function separable into temporal and spatial variables, and the frequency‑ or wavenumber‑dependent portion of the source function can be explicitly isolated via the inverse Fourier transform. By contrast, the source considered in the present section is characterized by a coupling between its spatial and frequency components, such that it can no longer be decoupled through the inverse Fourier transform. We still adopt the ideas from Section 3, for $\hat{\bx}\in \s^2$ and $k>0$, the far-field pattern of the electric field is 
\begin{equation*}
    \boe^\infty(\hat \bx,\om) =\mathrm{i}\om\mu(\mathbb I -\hat\bx\hat\bx^\mathsf{T}) \int_{t_{\min}}^{t_{\max}}\int_{\R^3} \mathrm{e}^{-\mathrm{i}\om c^{-1}(\hat{\bx}\cdot \by -t)} \bj(\by,t)\,\mathrm{d}\by \mathrm{d}t.
\end{equation*}

Analogous to the preceding text, for any fixed $\hat\bx\in\mathbb{S}^2$, choose the vertical vector $\bp^{(\hat\bx)}\in\mathbb{S}^2$ of $\hat\bx$. Then the inner product of $\bp^{(\hat\bx)}$ and the electric far-field pattern $\boe^{\infty}$ can be expressed as 
\begin{equation*}
    \bp^{(\hat\bx)}\cdot\boe^{\infty}(\hat \bx,\om)
    = \mathrm{i}\om\mu(\mathcal{F}\boh)(\om),
\end{equation*}
with
\begin{equation*}
   \begin{split}
       \boh(\alpha)&=\int_{t_{\min}}^{t_{\max}}{\int_{\Gamma(t+\alpha)} \bp^{(\hat\bx)}\cdot \bj(\by,t)}\mathrm{d}s(\by)\mathrm{d}t\\
       &=\int_{t_{\min}+\alpha}^{t_{\max}+\alpha}{\int_{\Gamma(t)} \bp^{(\hat\bx)}\cdot \bj(\by,t-\alpha)}\mathrm{d}s(\by)\mathrm{d}t.
   \end{split}
\end{equation*}

\begin{thm}\label{thm41}
    Define the region between two parallel hyperplanes as a slab $S(\hat\bx;\bj,t)$ and represent it as
    \begin{equation*}
        S(\hat\bx;\bj,t):=\{\by\in\mathbb{R}^3:c^{-1}\inf{(\hat\bx\cdot D)}-t_{\max}<\hat\bx\cdot\by<c^{-1}\sup{(\hat\bx\cdot D)}-t_{\min}\}.
    \end{equation*}
    If the set
    \begin{equation*}\label{zeroset}
     \{\alpha\in\mathbb{R}:\Gamma(t+\alpha)\subset S(\hat\bx;\bj,t),\boh(\alpha)=0\}
    \end{equation*}
    has Lebesgue measure zero, then the slab $S(\hat\bx;\bj,t)$ can be uniquely determined by the data  $\bp^{(\hat\bx)}\cdot\boe^{\infty}(\hat \bx,\om)$ for all $\om\in(\om_{\min},\om_{\max})$ at the observation direction $\hat\bx\in\mathbb{S}^2$.
\end{thm}

\begin{rem}
Clearly, we can deduce that 
\begin{equation*}
    \mathrm{supp}\boh(\alpha)\subset(c^{-1}\inf(\hat\bx\cdot D)-t_{\max},c^{-1}\sup(\hat\bx\cdot D)-t_{\min}).
\end{equation*}
\end{rem}

\begin{rem}
    The proof of Theorem \ref{thm41}, Theorem\ref{IndicatorI4.2} and Theorem \ref{IndicatorW43} closely resembles the proof of Theorem \ref{thm1}, Theorem \ref{Th3-2} and Theorem \ref{IndicatorW}, consequently, it has been omitted here for brevity. 
\end{rem}

\subsection{Indicator and test functions}\label{sec:4.2}

For any vector $\by\in\mathbb{R}^3$ and a fixed time instant $\eta$, we define the test function $\phi^{\hat{\bx}}_{\eta}(\by,\om)$ and the auxiliary indicator function $\mathcal I^{(\hat{\bx})}_{\eta}(\by,\om)$, both of which are identical to those described in Section \ref{sec:3.2}.

Define the unbounded parallel slabs which are  perpendicular to the observation direction $\hat\bx$
\begin{equation}\label{stripT_D}
    \begin{split}
        &T_{D,\eta}^{(\hat\bx)}:=\{\by\in\mathbb{R}^3: \inf(\hat\bx\cdot D)-c(t_{\max}-\eta)<\hat\bx\cdot\by<\sup(\hat\bx\cdot D)-c(t_{\min}-\eta)\},\\
        &T_{D,\eta}^{(-\hat\bx)}:=\{\by\in\mathbb{R}^3: \inf(\hat\bx\cdot D)+c(t_{\min}-\eta)<\hat\bx\cdot\by<\sup(\hat\bx\cdot D)+c(t_{\max}-\eta)\},
    \end{split}    
\end{equation}
We can still consider it as the result of translating $K_{D}^{(\hat\bx)}$ along or away from $\hat\bx$, respectively. Next we can draw the following conclusion.

\begin{thm}\label{IndicatorI4.2}
    Let $\hat{\bx}\in\mathbb{R}^2$ be fixed, we have
    \begin{equation*}
        \mathcal I^{(\hat{\bx})}_{\eta}(\by,\om)=\left
        \{\begin{array}{lll}
        \mbox{finite positive number}\quad&&\mbox{for} \quad \by\in T_{D,\eta}^{(\hat\bx)},\\
        0\quad&&\mbox{for}\quad \by\notin \overline{T_{D,\eta}^{(\hat\bx)}}.
       \end{array}\right.
    \end{equation*}
\end{thm}

Now we turn to recover the unknown excitation time $t_{\min}$ or termination time $t_{\max}$, to begin with the indicator function determined by a pair of opposite directions
\begin{equation*}
    \mathcal W_{\eta}^{(\hat{\bx})}(\by):=\left[\frac{1}{\mathcal I_{\eta}^{(\hat{\bx})}(\by)}+  \frac{1}{\mathcal I_{\eta}^{(-\hat{\bx})}(\by)}   \right]^{-1}=\frac{\mathcal I_{\eta}^{(\hat{\bx})}(\by)\; \mathcal I_{\eta}^{(-\hat{\bx})}(\by)}{\mathcal I_{\eta}^{(\hat{\bx})}(\by)+\mathcal I_{\eta}^{(-\hat{\bx})}(\by)}.
\end{equation*}
Clearly, Theorem \ref{IndicatorI4.2} and the definition in \eqref{stripT_D} yields the following theorem.
\begin{thm}\label{IndicatorW43}
    Let $\hat{\bx}\in\mathbb{R}^2$ be fixed, we have
    \begin{equation*}
        \mathcal W_{\eta}^{(\hat{\bx})}(\by,\om)=\left
        \{\begin{array}{lll}
        \mbox{finite positive number}\quad&&\mbox{for} \quad \by\in K_{D,\eta}^{(\hat\bx)}\cap T_{D,\eta}^{(-\hat\bx)},\\
        0\quad&&\mbox{for}\quad \by\notin \overline{K_{D,\eta}^{(\hat\bx)}\cap T_{D,\eta}^{(-\hat\bx)}}.
       \end{array}\right.
    \end{equation*}
\end{thm}

By plotting the one-dimensional function
\[
\eta \mapsto \max_{y\in B_R} \mathcal W_\eta^{(\hat x)}(y,\omega),
\]
we define
\[
T_t^{(\hat x)}(\eta) := \max_{y\in B_R} \mathcal W_\eta^{(\hat x)}(y,\omega).
\]

\begin{thm}\label{timeindicator2}
The function \(T_t^{(\hat x)}(\eta)\) satisfies
\[
T_t^{(\hat x)}(\eta)
=
\begin{cases}
\ge 0, & \eta \in [\eta_1,\eta_2], \\
0, & \eta \notin [\eta_1,\eta_2].
\end{cases}
\]
Moreover,
\[
t_{\max}=\eta_1+\eta_2-t_{\min}.
\]
\end{thm}

\begin{proof}
Recall that $\eta_1$ and $\eta_2$ denote the moments when 
$T_{D,\eta}^{(\hat\bx)}$ and $T_{D,\eta}^{(-\hat\bx)}$ begin to converge
and become completely separated, respectively.
By the definition of these transition points and the propagation
geometry of the signal, we obtain
\begin{equation*}
\begin{split}
\eta_1 &= t_{\min}-\frac{1}{2c}\big(\sup(\hat\bx\cdot D)-\inf(\hat\bx\cdot D)\big),\\
\eta_2 &= t_{\max}+\frac{1}{2c}\big(\sup(\hat\bx\cdot D)-\inf(\hat\bx\cdot D)\big).
\end{split}
\end{equation*}
Adding the two identities yields
\[
t_{\min}+t_{\max}=\eta_1+\eta_2,
\]
which completes the proof.
\end{proof}

\begin{rem}
Similarly, when the termination time $t_{\max}$ is known, the unknown excitation time $t_{\min}$ can be uniquely determined. In other words, whenever only one of the $t_{\min}$ or $t_{\max}$ is known, the other can always be uniquely determined by Theorem~\ref{timeindicator2}. If both $t_{\min}$ and $t_{\max}$ are unknown, a priori information on the diameter of the support of the source term in the direction perpendicular to some fixed direction is required to recover its location and the radiation interval. Otherwise, the proposed method is not applicable, and this case remains under investigation.
\end{rem}

After determining the source radiation time period $[t_{\min},t_{\max}]$, according to Theorem \ref{timeindicator2}, the smallest slab $K_{D}^{(\hat\bx)}$ containing
the source support can be determined by the indicator function $\mathcal S^{(\hat{\bx})}(\by,\om)$ using multi-frequency electricity far field pattern at a single direction, 
\begin{equation*}
    \mathcal S^{(\hat{\bx})}(\by,\om):=\left[\frac{1}{\mathcal I_{1}^{(\hat{\bx})}(\by)}+  \frac{1}{\mathcal I_{2}^{(-\hat{\bx})}(\by)}   \right]^{-1}=\frac{\mathcal I_{1}^{(\hat{\bx})}(\by)\; \mathcal I_{2}^{(-\hat{\bx})}(\by)}{\mathcal I_{1}^{(\hat{\bx})}(\by)+\mathcal I_{2}^{(-\hat{\bx})}(\by)},
\end{equation*}
where
\begin{equation}\label{IndicatorIII42}
    \mathcal I_j^{(\hat{\bx})}(\by,\om) =\int_{\R}\mathbb{E}^{\infty}(\hat \bx,\om)\overline{\phi_j^{(\hat{\bx})}(\by,\om)} \mathrm{d}{\om},\quad j=1,2,
\end{equation}
with test functions $\phi_1^{(\hat{\bx})}(\by,\om):=\mathrm{e}^{\mathrm{i}\om(c^{-1}\hat\bx\cdot \by-t_{\min})}$, $\phi_2^{(\hat{\bx})}(\by,\om):=\mathrm{e}^{\mathrm{i}\om(c^{-1}\hat\bx\cdot \by-t_{\max})}$, $\by\in\mathbb{R}^3$.

Clearly, the indicator function $\mathcal S^{(\hat{\bx})}(\by,\om)$ fulfills that
\begin{equation*}
    \mathcal S^{(\hat{\bx})}(\by,\om)=\left
    \{\begin{array}{lll}
    \mbox{finite positive number}\quad&&\mbox{for} \quad \by\in K_{D}^{(\hat\bx)},\\
     0\quad&&\mbox{for}\quad \by\notin \overline{K_{D}^{(\hat\bx)}}.
     \end{array}\right.
\end{equation*}
With multiple but sparse observation directions, the shape and location of the source support can be reconstructed.

Analogous to the preceding section, after determining the smallest slab containing the source support for a given observation direction, the geometry and location of the source support can be recovered from multiple but sparse observation directions. Since the derivation follows exactly the same steps as above, we omit here.

\section{Numerical examples}
In this section, we present several three-dimensional numerical experiments to demonstrate the performance of the proposed direct sampling method for reconstructing the support of electromagnetic sources from multi-frequency far-field measurements. The numerical experiments are designed to verify the theoretical results established in Section 3, including the recovery of shifted slabs, the determination of the excitation time, and the reconstruction of the source support.

For simplicity, we restrict our attention to the inverse problem IP1, where the source emits an impulsive signal at an unknown excitation time $t_0$.

Synthetic multi-frequency electric far-field data are generated according to the analytical representation derived in \eqref{e-infty}. The indicator functions are evaluated on a uniform sampling grid covering a bounded reconstruction domain. The reconstructed regions are visualized using slices and iso-surfaces of the corresponding indicator functions.

\subsection{Frequency discretization and implementation details}

In the theoretical formulation of the indicator function, the frequency
variable is integrated over the whole real axis, i.e.,
$\omega \in (-\infty,+\infty)$.
However, in practical computations only a finite frequency band is available.
Therefore, the integration over $(-\infty,+\infty)$ is truncated to a bounded interval
$(-W,W)$.

The function $\omega \mapsto \boe^{\infty}(\hat x,\omega)$ admits an analytic
extension to negative frequencies via the Hermitian symmetry
\[
\boe^{\infty}(\hat x,-\omega)
=
\overline{\boe^{\infty}(\hat x,\omega)},
\qquad \omega>0.
\]
Using this symmetry (and the same property for the test function),
the indicator function \eqref{Indicator}
is approximated on the truncated frequency band as
\begin{align*}
\mathcal I^{(\hat{\bx})}_{\eta}(\by)
&\approx 
\int_{-W}^{W}
\dfrac{1}{\mathrm{i}\omega\mu} 
\, \bp^{(\hat\bx)}\cdot\boe^{\infty}(\hat{\bx},\omega)
\,\overline{\phi_{\eta}^{(\hat{\bx})}(\by,\omega)}
\, d\omega
\\
&=
\int_{0}^{W}
\dfrac{1}{\mathrm{i}\omega\mu} 
\, \bp^{(\hat\bx)}\cdot\boe^{\infty}(\hat{\bx},\omega)
\,\overline{\phi_{\eta}^{(\hat{\bx})}(\by,\omega)}
\, d\omega
+
\int_{-W}^{0}
\dfrac{1}{\mathrm{i}\omega\mu} 
\, \bp^{(\hat\bx)}\cdot\boe^{\infty}(\hat{\bx},\omega)
\,\overline{\phi_{\eta}^{(\hat{\bx})}(\by,\omega)}
\, d\omega
\\
&=
\int_{0}^{W}
\dfrac{1}{\mathrm{i}\omega\mu} 
\, \bp^{(\hat\bx)}\cdot\boe^{\infty}(\hat{\bx},\omega)
\,\overline{\phi_{\eta}^{(\hat{\bx})}(\by,\omega)}
\, d\omega
+
\int_{0}^{W}
\overline{
\dfrac{1}{\mathrm{i}\omega\mu} 
\, \bp^{(\hat\bx)}\cdot\boe^{\infty}(\hat{\bx},\omega)
}
\, \phi_{\eta}^{(\hat{\bx})}(\by,\omega)
\, d\omega
\\
&=
2\,\mathrm{Re}
\left\{
\int_{0}^{W}
\dfrac{1}{\mathrm{i}\omega\mu} 
\, \bp^{(\hat\bx)}\cdot\boe^{\infty}(\hat{\bx},\omega)
\,\overline{\phi_{\eta}^{(\hat{\bx})}(\by,\omega)}
\, d\omega
\right\}.
\end{align*}

In our numerical experiments, we choose the truncated frequency interval
\[
\omega \in (0,20),
\]
which is uniformly discretized into $N_\omega = 200$ sampling points.
More precisely,
\[
\omega_n = n\Delta\omega,
\qquad
n=1,2,\ldots,200,
\]
with
\[
\Delta\omega = \frac{20}{200}.
\]

The truncation from $(-\infty,+\infty)$ to $(-W,W)$ and the subsequent
restriction to positive frequencies, together with the uniform discretization,
lead to a practical implementation of the proposed indicator function.
All indicator functions are evaluated using this discrete formulation
on a uniform Cartesian sampling grid.

Since the reconstruction procedure only requires explicit evaluations of
the indicator functions and does not involve solving additional forward
problems, the overall computational cost grows linearly with respect to
both the number of sampling points and the number of discrete frequencies.

\subsection{Recovery of the shifted slabs $K_{D,\eta}^{(\hat x)}$ and  $K_{D,\eta}^{(-\hat x)}$ in $\R^3$}


We first verify Theorem 3.2 by reconstructing the shifted slabs
\[
K_{D,\eta}^{(\hat x)}
\quad \text{and} \quad
K_{D,\eta}^{(-\hat x)},
\]
using multi-frequency far-field data collected from a single observation direction.
Throughout this subsection, the observation direction is chosen as
$\hat x=(0,0,1),$
and the source support $D$ is assumed to be cube-shaped.

Figures~1 and~2 present the reconstructions of
$K_D^{(\hat x),\eta}$ and $K_D^{(-\hat x),\eta}$,
corresponding to the observation directions
$\hat x=(0,0,1)$ and $-\hat x=(0,0,-1)$, respectively.
Different electromagnetic parameters $(\mu,\varepsilon)$ and shifting moments $\eta$ are considered.
For visualization, we display the cross-sections on the plane $y_1oy_3$.

In Figure~1, the three subfigures (a)--(c) correspond to
\[
-(\varepsilon\mu)^{-1/2}(t_0-\eta)
=
\begin{cases}
0, & \text{in (a)},\\
0.5, & \text{in (b)},\\
1, & \text{in (c)}.
\end{cases}
\]
As $\eta$ varies, the reconstructed slab is translated along the positive
$\hat x$-direction with increasing shift magnitude.

In Figure~2, the three subfigures (a)--(c) correspond to
\[
(\varepsilon\mu)^{-1/2}(t_0-\eta)
=
\begin{cases}
0, & \text{in (a)},\\
-0.5, & \text{in (b)},\\
-1, & \text{in (c)}.
\end{cases}
\]
In this case, the reconstructed slab is translated along the negative
$\hat x$-direction, in full agreement with the theoretical prediction.

In both figures, the numerical shifts coincide precisely with the values
predicted by Theorem~3.2. The reconstructed regions clearly exhibit
slab-like structures perpendicular to the observation direction.
Moreover, varying $\eta$ induces only a rigid translational shift of the slab,
while its geometric profile remains unchanged.
This confirms that the proposed indicator function correctly captures
the directional projection of the source support
from single-direction observation data.

\begin{figure}[H]
	\centering
	\subfigure[$\mu=1; \epsilon=1, t_0=1, \eta=1$]{
		\includegraphics[scale=0.22]{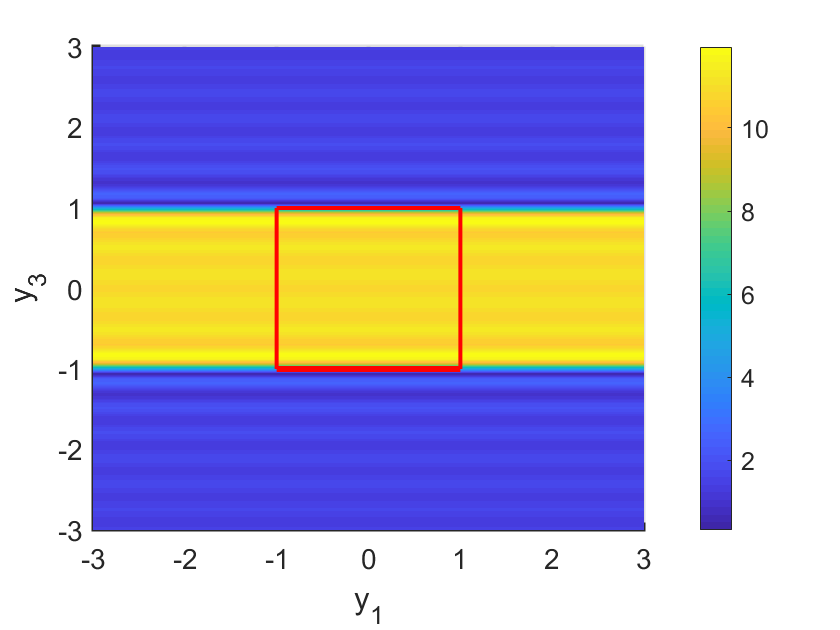}
	}
	\subfigure[$\mu=2; \epsilon=2, t_0=1, \eta=2;
$]{
		\includegraphics[scale=0.22]{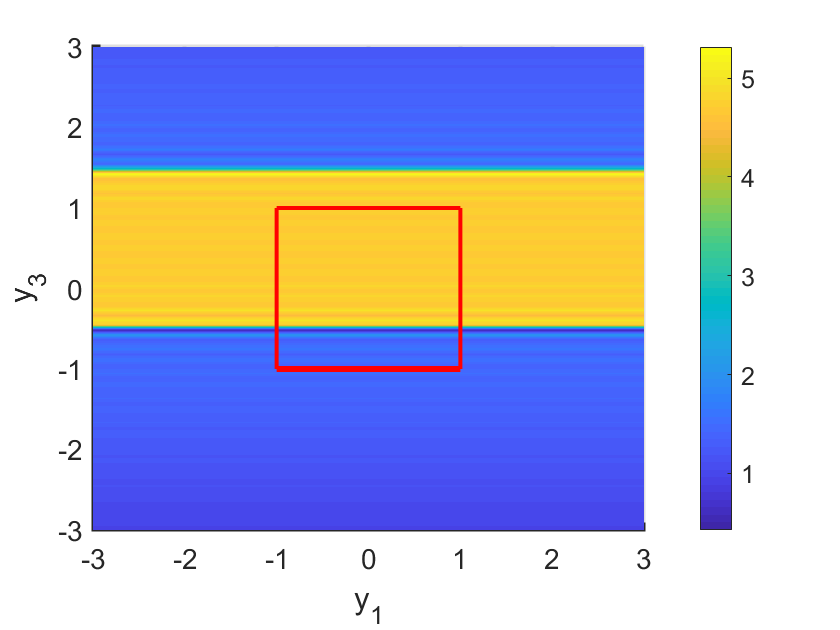}
		
	}
	\subfigure[$\mu=4; \epsilon=1, t_0=1, \eta=3$]{
		\includegraphics[scale=0.22]{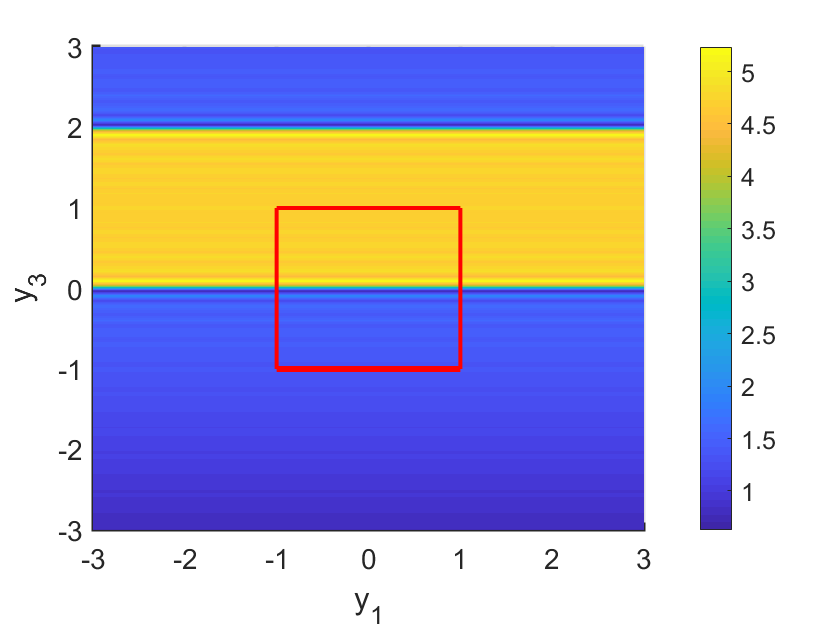}
		
	}
	\caption{Reconstruction of $K_{D,\eta}^{(\hat x)}$ using multi-frequency data from a single observation direction $\hat x=(0,0,1)$, the source is supported in a cube-shaped.
	} \label{fig:3d-rec}
\end{figure}

\begin{figure}[H]
	\centering
	\subfigure[$\mu=1; \epsilon=1, t_0=1, \eta=1$]{
		\includegraphics[scale=0.22]{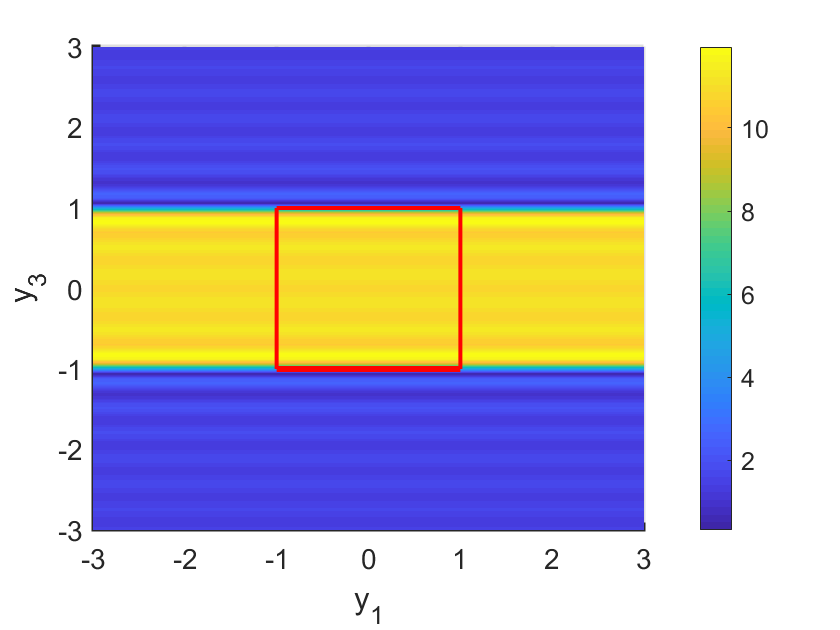}
	}
	\subfigure[$\mu=2; \epsilon=2, t_0=1, \eta=2;
$]{
		\includegraphics[scale=0.22]{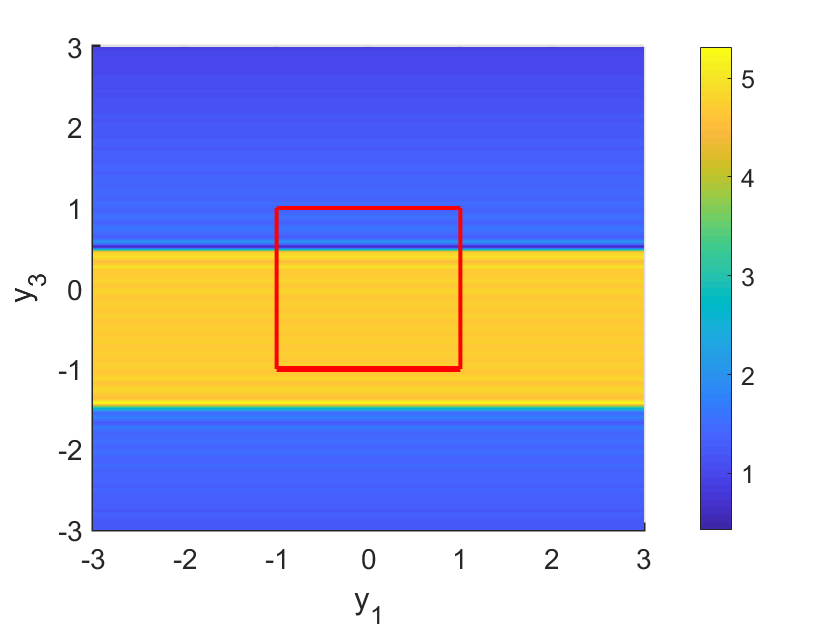}
		
	}
	\subfigure[$\mu=4; \epsilon=1, t_0=1, \eta=3$]{
		\includegraphics[scale=0.22]{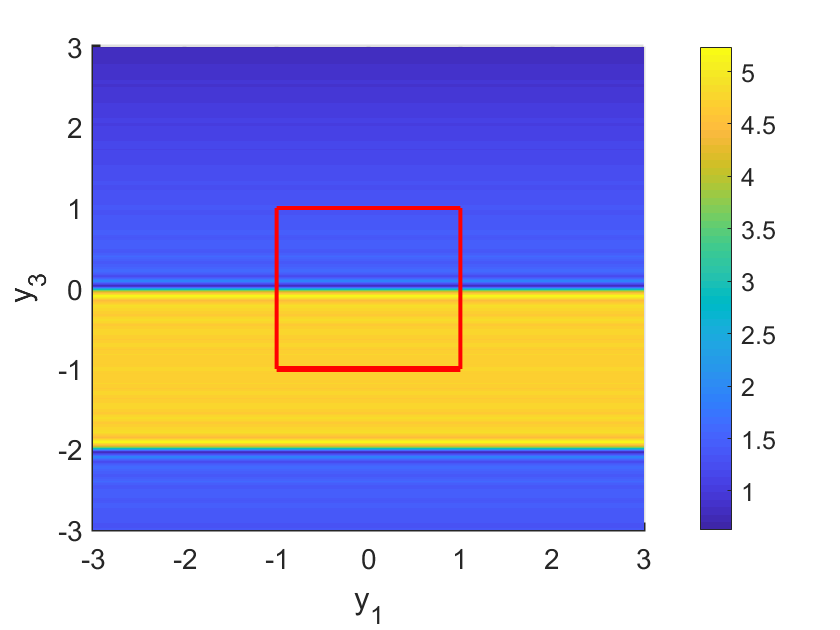}
		
	}
	\caption{Reconstruction of $K_{D,\eta}^{(-\hat x)}$ using multi-frequency data from a single observation direction $-\hat x=(0,0,-1)$, the source is supported in a cube-shaped.
	} \label{fig:3d-rec}
\end{figure}

Figures~3(a)--3(d) illustrate the reconstruction of the intersection
region
\[
K_{D,\eta}^{(\hat x)} \cap K_{D,\eta}^{(-\hat x)},
\]
which provides a numerical verification of Theorem~3.3.
According to the theorem, the indicator function
$W_\eta^{(\hat x)}(y)$ takes positive values precisely inside
the intersection of the two shifted slabs and vanishes outside it.

When the shifting parameter $\eta$ is small,
the slabs $K_{D,\eta}^{(\hat x)}$ and
$K_{D,\eta}^{(-\hat x)}$ are spatially separated,
and therefore their intersection is empty or very limited,
as shown in Figure~3(a).
As $\eta$ increases, the two slabs gradually move toward each other,
leading to an expanding overlap region.
When $\eta$ approaches the correct excitation moment,
a maximal intersection region appears,
which corresponds to the situation predicted by
Theorem~3.3; see Figure~3(b). As $\eta$ increases further away from the correct excitation time,
the slabs $K_{D,\eta}^{(\hat x)}$ and
$K_{D,\eta}^{(-\hat x)}$ move apart again in opposite directions.
Consequently, their overlap region gradually shrinks
and eventually vanishes, as illustrated in Figures~3(c)--3(d).

These numerical results clearly demonstrate that the proposed
indicator function successfully characterizes the intersection
structure of the shifted slabs.
The evolution of the overlap region with respect to $\eta$
provides a reliable numerical mechanism for determining the
unknown excitation time.

\begin{figure}[H]
	\centering
	\subfigure[$t_0=1, \eta=0$]{
		\includegraphics[scale=0.22]{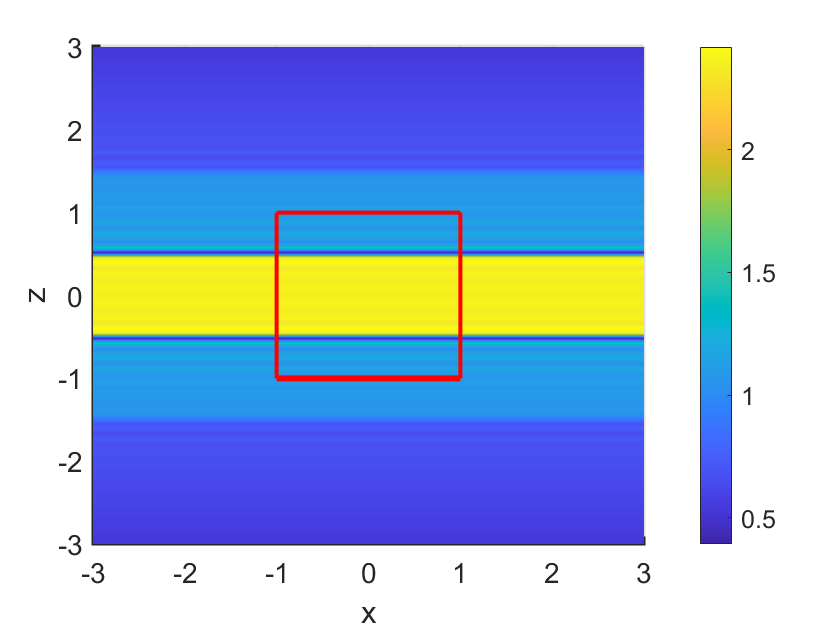}
	}
	\subfigure[$t_0=1, \eta=1;
$]{
		\includegraphics[scale=0.22]{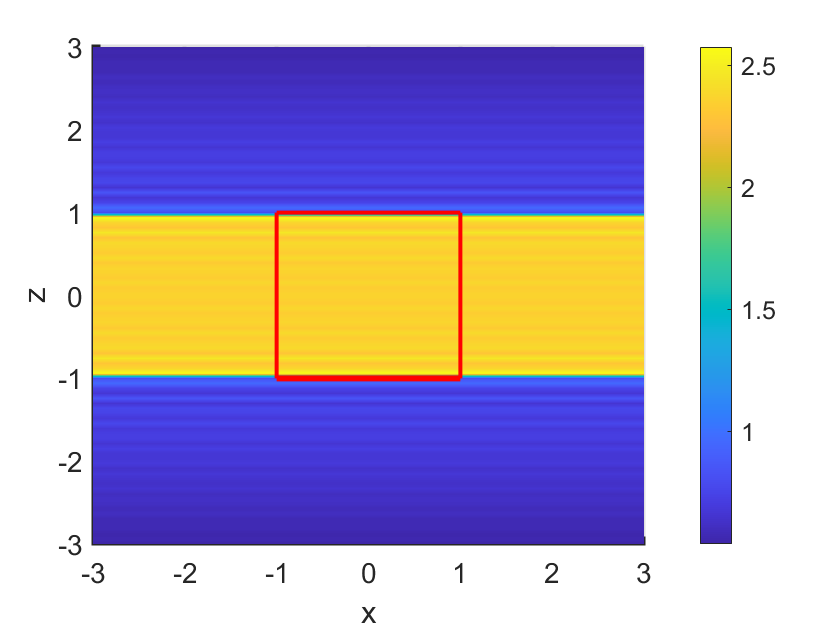}
		
	}
    \\
	\subfigure[$ t_0=1, \eta=2$]{
		\includegraphics[scale=0.22]{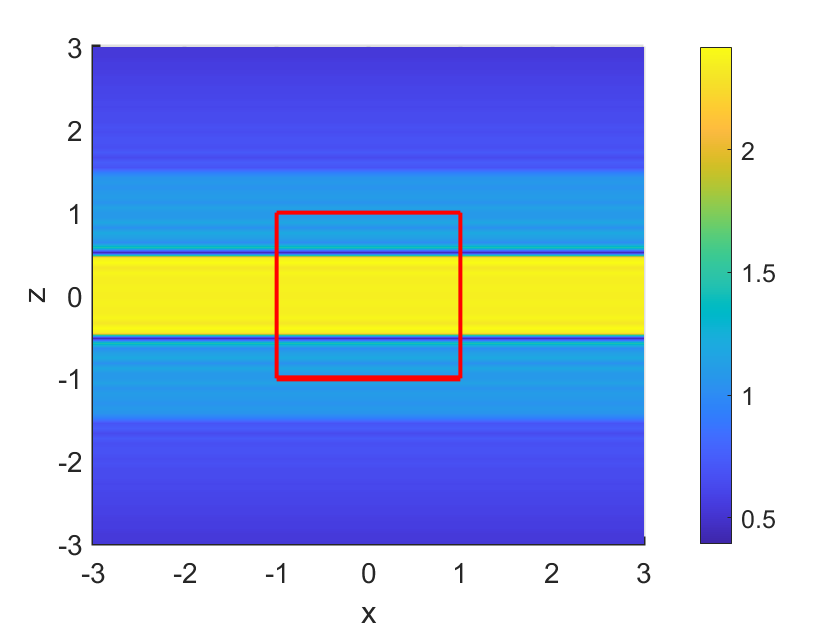}
		
	}
    \subfigure[$ t_0=1, \eta=3$]{
		\includegraphics[scale=0.22]{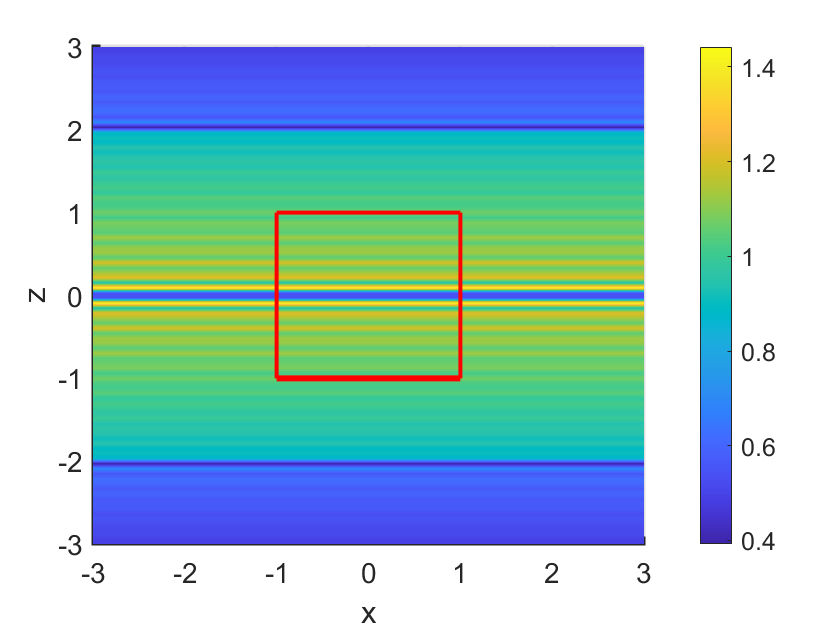}
		
	}
	\caption{Reconstruction of $K_{D,\eta}^{(\hat x)}\bigcap K_{D,\eta}^{(-\hat x)}$using multi-frequency data from a single observation direction $\hat x=(0,0,1),\mu=4,\epsilon=1$, the source is supported in a cube-shaped.
	} \label{fig:3d-rec}
\end{figure}

\subsection{Determination of the excitation time $t_0$ and reconstruction of source supports}

Next, we demonstrate the recovery of the unknown excitation time
$t_0$ based on Theorem~3.4. In practice, the excitation time is determined by evaluating
\[
\mathcal T_{t_0}^{(\hat x)}(\eta)
=\max_{y\in B_R} \mathcal W_\eta^{(\hat x)}(y),
\]
from which the interval $[\eta_1,\eta_2]$ is detected numerically
and the excitation time is recovered via
\[
t_0=\frac{\eta_1+\eta_2}{2}.
\]

We next consider several examples with different source geometries
to illustrate the performance of the proposed reconstruction method.

\paragraph{Example 1: Cubic source.} The source support is chosen as a cube with excitation time $t_0=3$.
Figure~4(a) displays the original cubic source support.
Figure~4(b) presents the reconstruction of the excitation time
obtained from the indicator function
$\mathcal T_{t_0}^{(\hat x)}(\eta)$.
The plateau region of the curve allows one to determine
the interval $[\eta_1,\eta_2]$, from which the excitation
time is accurately recovered as $t_0=3$.
After determining the excitation moment,
the source support is reconstructed by combining
multi-frequency far-field data from three mutually
orthogonal observation directions parallel to the
coordinate axes.
The reconstructed support is shown in Figure~4(c),
where an appropriate iso-surface value is selected.
The recovered region accurately captures both the
location and the cubic geometry of the true source support.

Figure~5 further presents reconstructions obtained from single
observation direction, as $(1,0,0),\quad (0,1,0),\quad (0,0,1).$ 
By combining information from multiple directions,
the reconstructed region provides an accurate approximation
of the convex hull of the source support; as shown in Figure~5(c).

\paragraph{Example 2: Spherical source.}
In this example, the source support is chosen as a ball with excitation
time $t_0=4$. The reconstruction shown in Figure~6 demonstrates that
the proposed method performs stably for smoothly curved geometries,
and the reconstructed iso-surface agrees well with the exact support.

\paragraph{Example 3: Ellipsoidal source.}
Finally, an ellipsoidal source support with excitation time $t_0=5$
is considered; see Figure~7.
Despite the anisotropic geometry, the reconstructed support
successfully captures the elongation and orientation of the true
source region, indicating that the proposed method is insensitive
to the specific shape of the source.

\begin{figure}[H]
	\centering
	\subfigure[A cubic support]{
		\includegraphics[scale=0.22]{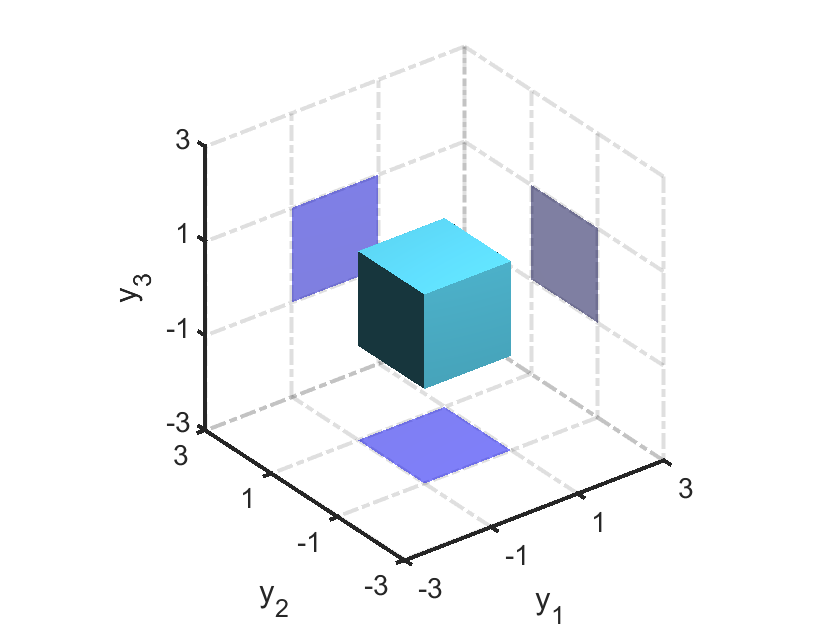}
	}
	\subfigure[$ t_0=3$]{
		\includegraphics[scale=0.18]{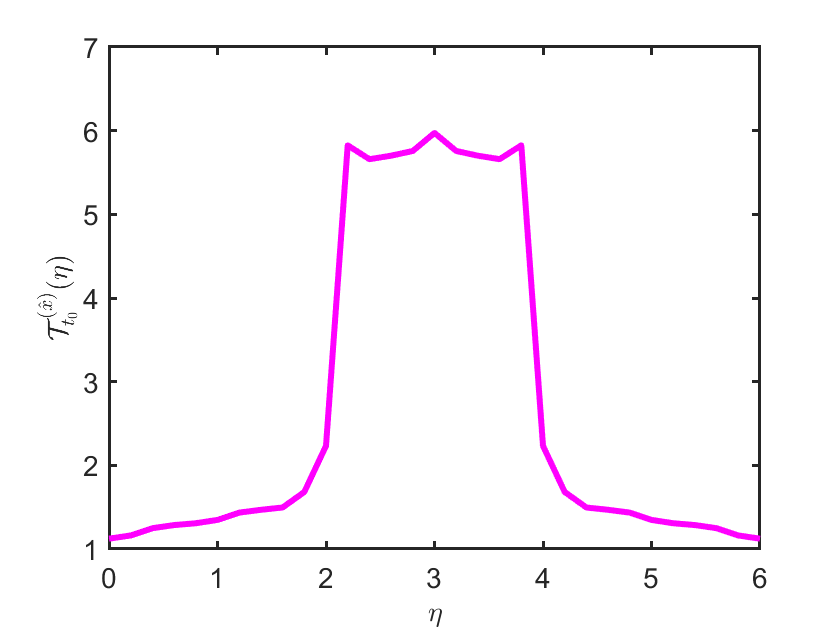}
		
	}
	\subfigure[Iso-surface value $=1$]{
		\includegraphics[scale=0.22]{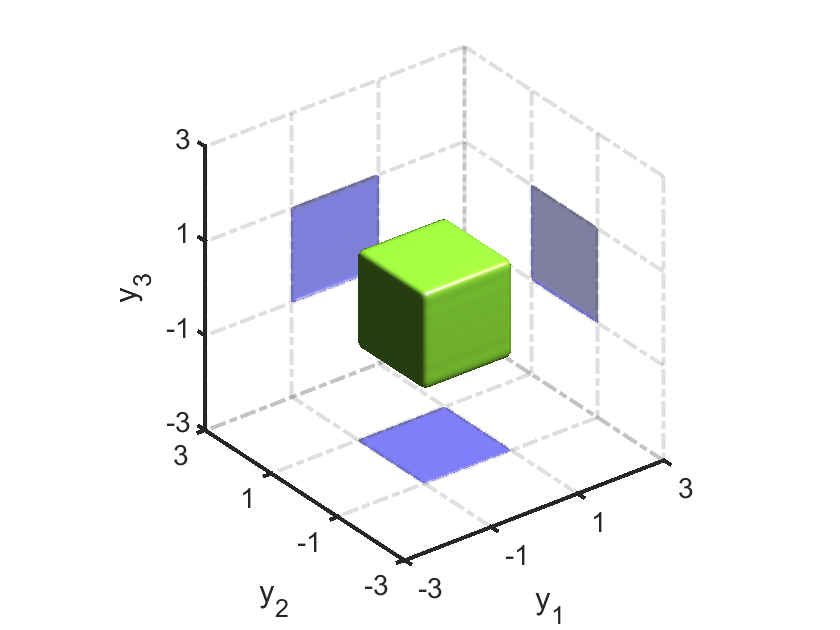}
		
	}

	\caption{Reconstruction of a cubic support.  $\mu=1, \epsilon=1,t_0=3$.
	} \label{fig:3d-rec}
\end{figure}

\begin{figure}[H]
	\centering
    	\subfigure[$\hat x =(1 \;0 \;0)$]{
		\includegraphics[scale=0.2]{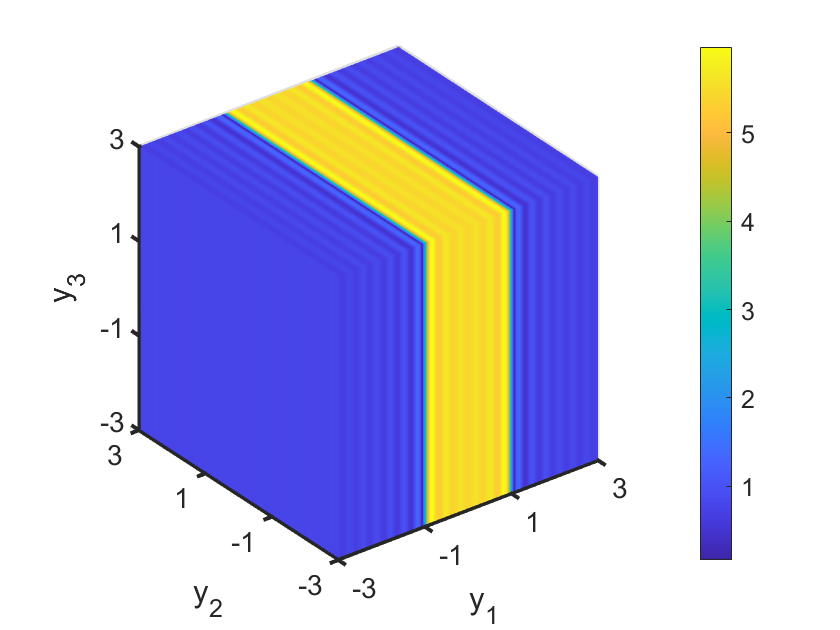}
	}
	\subfigure[$\hat x =(0\;1 \;0)$]{
		\includegraphics[scale=0.2]{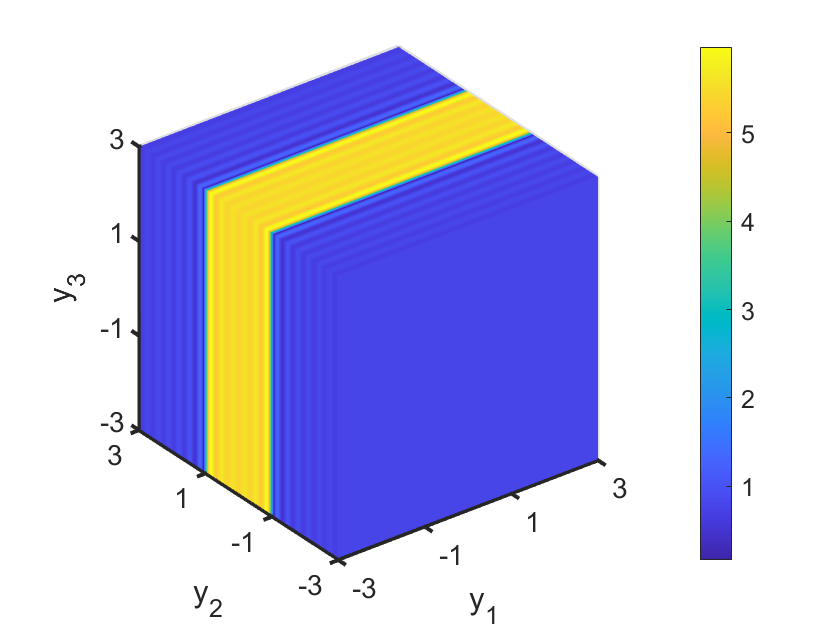}
		
	}
	\subfigure[$\hat x =(0 \;0 \;1)$]{
		\includegraphics[scale=0.2]{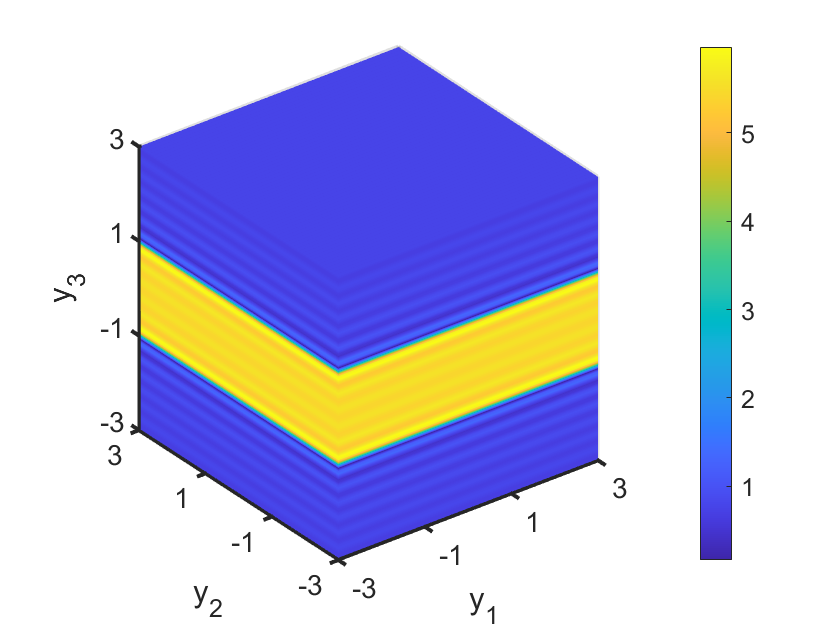}
		
	}
	\caption{Reconstruction of $K_{D,\eta}^{(\hat x)} \cap K_{D,\eta}^{(-\hat x)}$.  $\mu=1, \epsilon=1,t_0=3$.
	} \label{fig:3d-rec}
\end{figure}

\begin{figure}[H]
	\centering
	\subfigure[A ball support]{
		\includegraphics[scale=0.22]{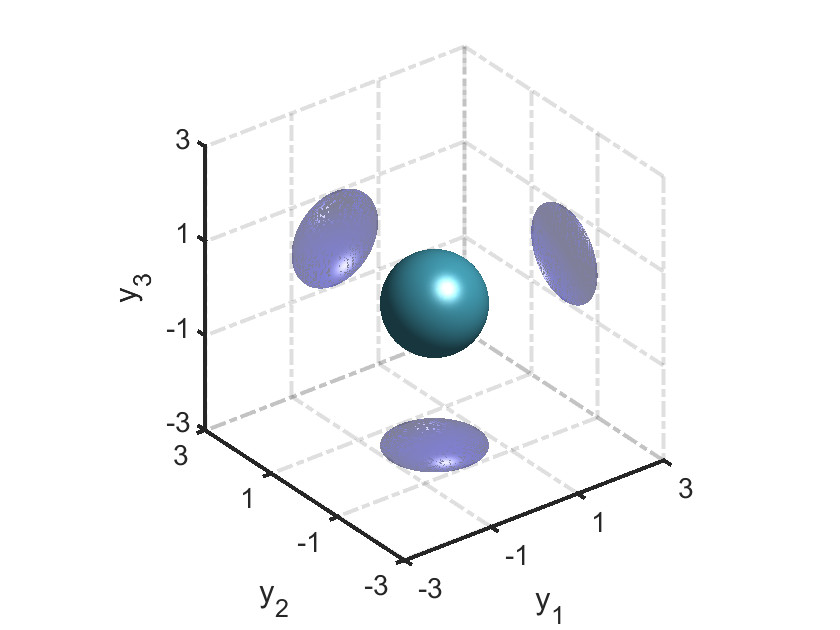}
	}
	\subfigure[$ t_0=4$]{
		\includegraphics[scale=0.18]{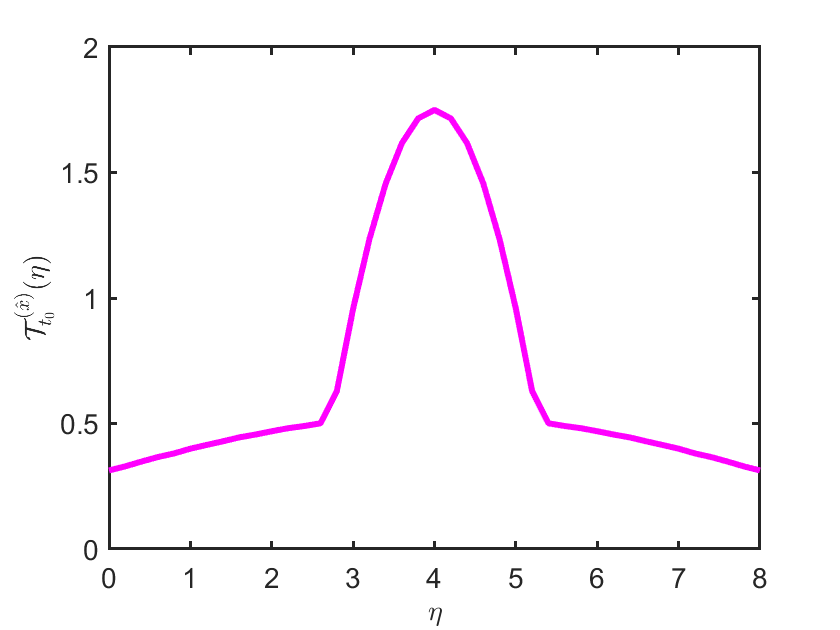}
		
	}
	\subfigure[Iso-surface  $=0.06$ and $M=15$]{
		\includegraphics[scale=0.22]{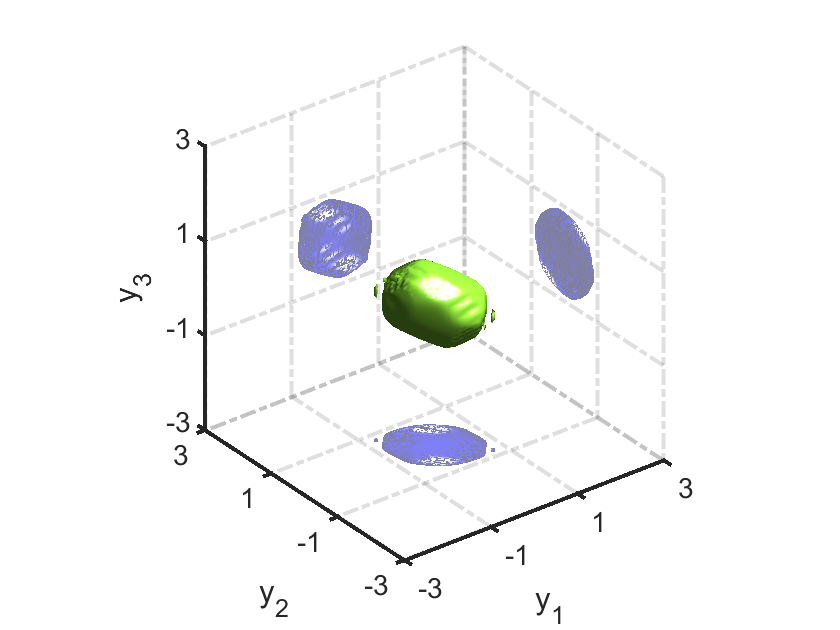}
		
	}

	\caption{Reconstruction of a ball-shaped support.  $\mu=1, \epsilon=1,t_0=4$.
	} \label{fig:3d-rec}
\end{figure}

\begin{figure}[H]
	\centering
	\subfigure[A elliptic support]{
		\includegraphics[scale=0.22]{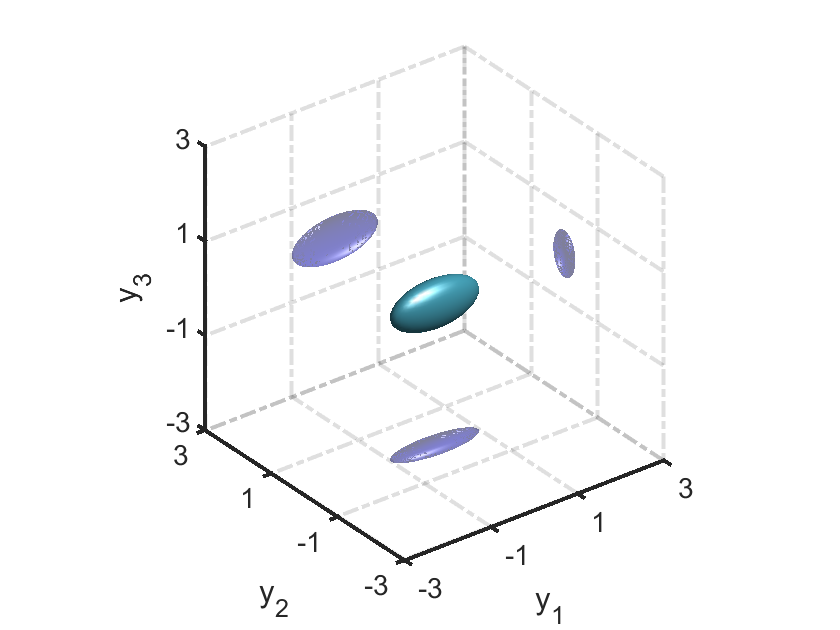}
	}
	\subfigure[$ t_0=5$]{
		\includegraphics[scale=0.18]{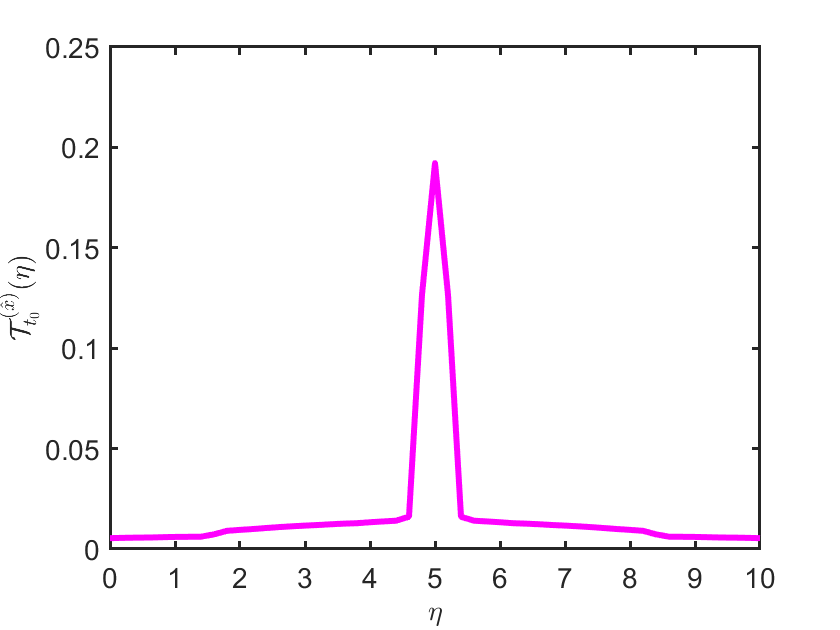}
		
	}
	\subfigure[Iso-surface $=0.006$, $M=20$]{
		\includegraphics[scale=0.22]{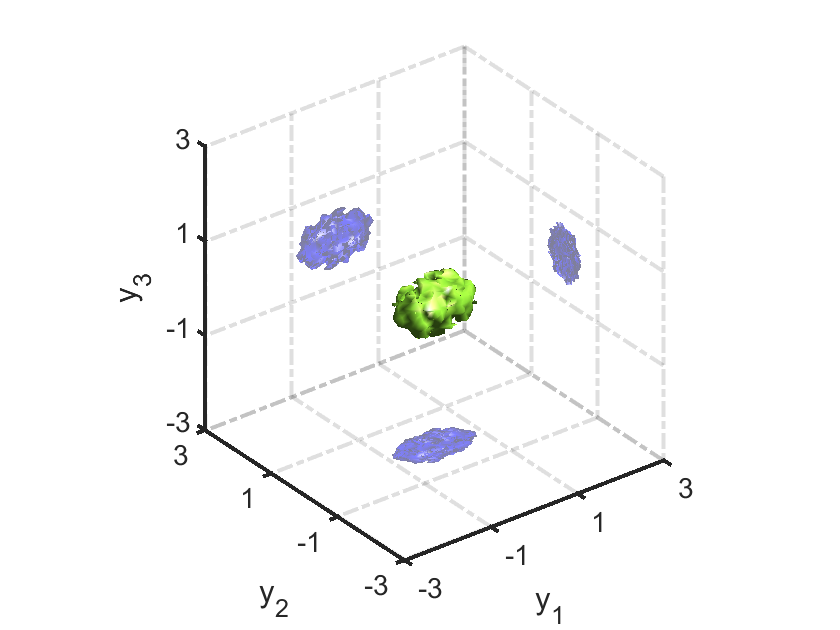}
		
	}

	\caption{Reconstruction of an elliptic support.  $\mu=1, \epsilon=1,t_0=5$.
	} \label{fig:3d-rec}
\end{figure}

\subsection{Robustness with respect to measurement noise}

To investigate the stability of the proposed method,
noisy far-field data are generated according to
\[
E^\infty_\delta(\hat x,\omega)
=
E^\infty(\hat x,\omega)
\bigl(1+\delta\,\xi\bigr),
\]
where $\delta>0$ denotes the relative noise level and
$\xi\sim\mathcal N(0,1)$ is Gaussian noise.

Numerical experiments indicate that the reconstructed slabs
remain stable under moderate noise levels.
In particular, the location of the reconstructed support and
the recovered excitation time are only mildly affected,
while noise mainly reduces the sharpness of reconstructed boundaries.

This stability can be attributed to the intrinsic averaging effect
induced by the multi-frequency integration in the indicator functions,
which suppresses random perturbations in the measurement data.

\begin{figure}[H]
	\centering
	\subfigure[$\delta=0$]{
		\includegraphics[scale=0.22]{FIG/2-1.png}
	}
       \subfigure[$\delta=30\%$]{
		\includegraphics[scale=0.22]{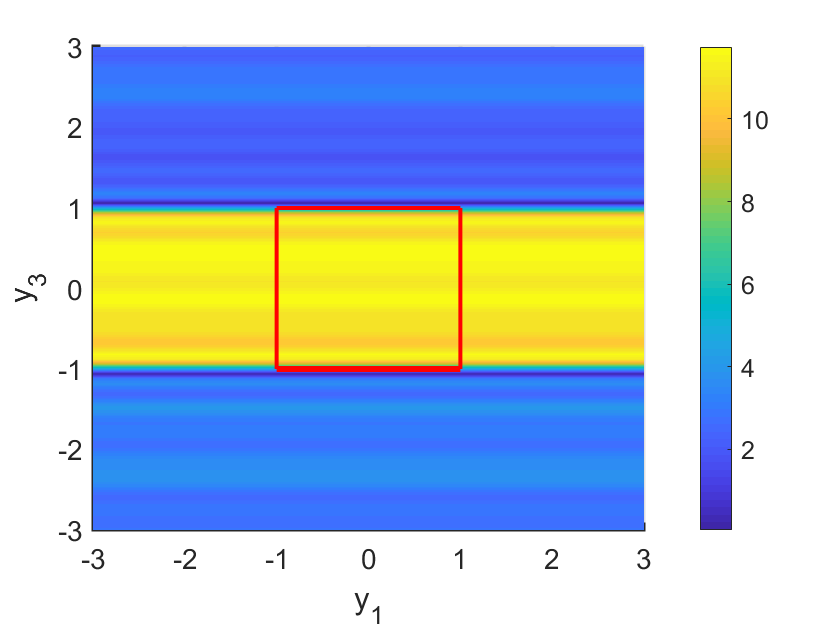}
	}

	\subfigure[$\delta=50\%$]{
		\includegraphics[scale=0.22]{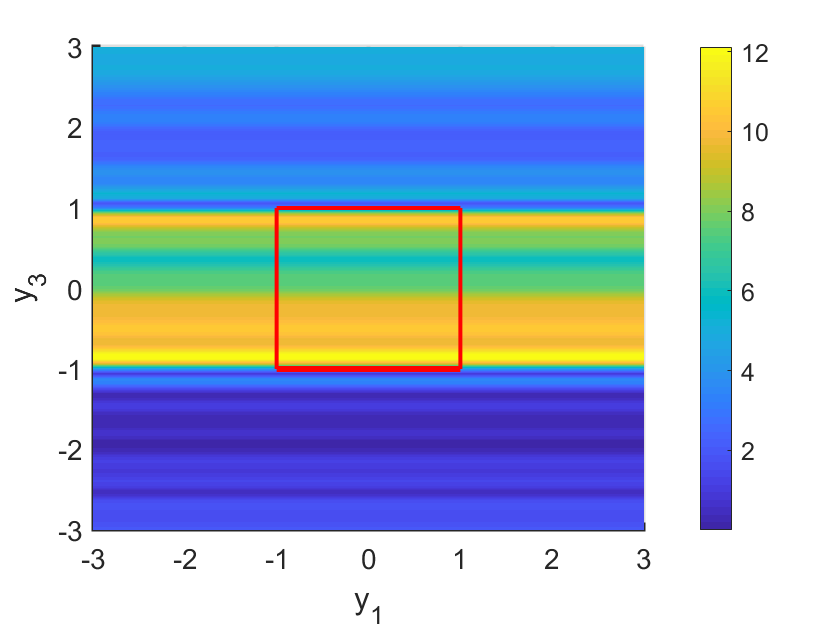}

	}
	\subfigure[$\delta=80\%$]{
		\includegraphics[scale=0.22]{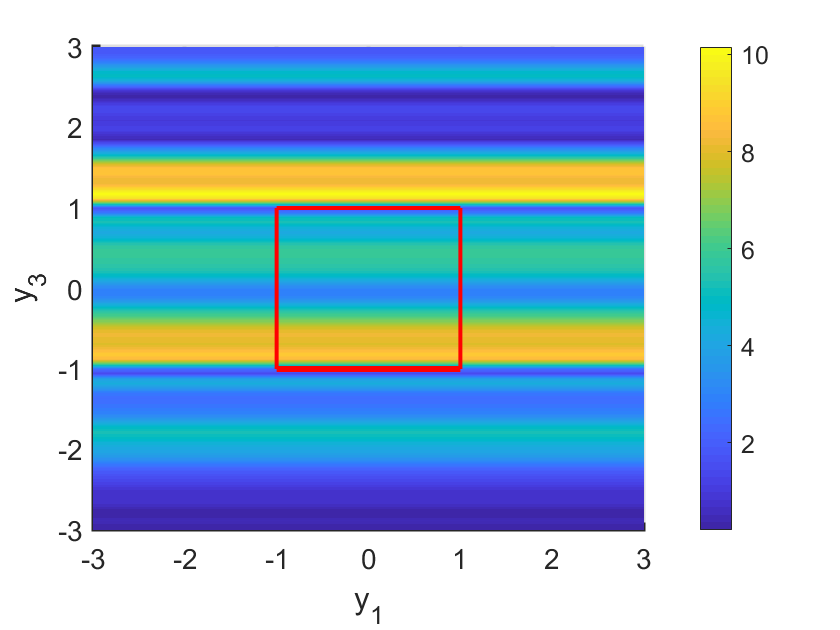}
		
	}

	\caption{Reconstruction of a cube-shaped support with multi-frequency data from single
 observation direction at different noise levels $\delta$.
	} \label{fig:3d-rec}
\end{figure}

To investigate the stability of the proposed method,
noisy far-field data are generated according to
\[
E^\infty_\delta(\hat x,\omega)
=
E^\infty(\hat x,\omega)
\bigl(1+\delta\,\xi\bigr),
\]
where $\delta>0$ denotes the relative noise level and
$\xi\sim\mathcal N(0,1)$ represents complex Gaussian noise.

Figure~8 presents the reconstruction results for different
noise levels. Figure~8(a) corresponds to the noise-free case
($\delta=0$), where the reconstructed slab clearly matches
the exact projection region of the cubic source support.
Figure~8(b)--(d) display the results with
$\delta=30\%$, $50\%$, and $80\%$, respectively.

It can be observed that for moderate noise levels
($\delta=30\%$ and $50\%$),
the slab structure remains clearly identifiable and
the location of the reconstructed region is only slightly affected.
Even when the noise level reaches $\delta=80\%$,
the principal slab structure is still visible,
although boundary oscillations become more pronounced.

The strong robustness of the proposed method can be explained
by the intrinsic averaging effect induced by the frequency
integration in the indicator function.
Since the reconstruction is based on integrating the
multi-frequency data against oscillatory test functions,
random perturbations in the measurements tend to cancel out,
while the coherent structural information associated with
the true source support is reinforced.
This mechanism significantly enhances the noise tolerance
of the method.

\section{Conslusion}
In this work, we extend the direct sampling method to the time-domain electromagnetic inverse source problem governed by the vector Maxwell equations by developing a frequency-domain approach. Using sparse far-field data, we first accurately determine the excitation time of the signal, and then reconstruct the convex hull of the source support. Moreover, the proposed approach remains valid for sources radiating continuously over a finite time interval, provided that either the initiation time or termination time of the radiation is known. Notably, it does not require the source function to be divergence-free. Finally, we consider this approach suitable for elastic waves and capable of being generalized to inverse problems involving extended moving sources.

\newpage


\newpage

\section*{Acknowledgements}

The work of Fenglin Sun was supported by the NSFC (No.12271404) and the work of Hongxia Guo was supported by NSFC  (No.12501591) and Natural Science Foundation of Tianjin (No.24JCQNJC00850). 
The authors wish to express particular gratitude to Professor Guanghui Hu for his insightful and detailed discussions, which have contributed to notable improvements in the results and presentation of this work.


\begin{thebibliography}{00}
\bibliographystyle{unsrt}
\bibitem {Ammari2002} G. Bao, H. Ammari,  and J. Fleming, An inverse source problem for Maxwell’s equation in
magnetoencephalography, 
SIAM Journal on Applied Mathematics, 62 (2002), 1369-1382.

\bibitem{Balanis2005} C. A. Balanis, Antenna Theory: Analysis and Design, Wiley, Hoboken, NJ, 2005.

\bibitem{Dassios2005} G. Dassios, A. Fokas and F. Kariotou, On the non-uniqueness of the inverse MEG problem, Inverse Problems, 21(2005): L1-5.


\bibitem{Bleistein1977} N. Bleistein and J. Cohen, Nonuniqueness in the inverse source problem in acoustics and electromagnetics, Journal of Mathematical Physics, 18(1977): 194-201.

\bibitem {Albanese2006} R. Albanese,  and P. Monk, The inverse source problem for Maxwell's equations, Inverse problems 22(2006): 1023.

\bibitem {Valdivia2012}N P. Valdivia, Electromagnetic source identification using multiple frequency information, Inverse Problems, 28(2012):115002.

\bibitem {Li2024} P. Li and J. Wang, Nonradiating sources of Maxwell's equations, arxiv:2402.10407 (2024).

\bibitem {BaoAmmari2002} G. Bao, H. Ammari and J. Fleming, An inverse source problem for Maxwell’s equation in magnetoencephalography, SIAM Journal on Applied Mathematics, 62(2002): 1369-1382.

\bibitem {Bao2010} G. Bao, J.Lin and F. Triki, A multi-frequency inverse source problem, Journal of Differential Equations, 249(2010): 3443-3465.

\bibitem {BaoLi2015} G. Bao, P. Li, J. Lin and F. Triki, Inverse scattering problems with multi-frequencies, Inverse Problems, 31(2015):093001.

\bibitem {Liu2015} H. Liu and G. Uhlmann, Determining both sound speed and internal source in thermo- and photo-acoustic tomography, Inverse Problems, 31(2015):105005.

\bibitem{CIL} J. Cheng, V. Isakov and S. Lu,  Increasing stability in the inverse source problem with many frequencies, Journal of Differential Equations, 260 (2016): 4786-4804.


\bibitem {Bao2020} G. Bao, P. Li and Y. Zhao, Stability for the inverse source problems in elastic and electromagnetic waves, Journal de Math\'{e}matiques Pures et Appliqu\'{e}es 134 (2020): 122-178.

\bibitem {LiYamamoto2005} S. Li and M. Yamamoto, An inverse source problem for Maxwell's equations in anisotropic media, Applicable Analysis 84(2005): 1051-1067.

\bibitem{HLLZ} G. Hu, P. Li, X. Liu and Y. Zhao,  Inverse source problems in electrodynamics, Inverse Problems and Imaging, 12(2018): 1411--1428.

\bibitem{Hukian2019} G. Hu, Y. Kian, P. Li and Y. Zhao, Inverse moving source problems in electrodynamics, Inverse Problems, 35(2019):075001.

\bibitem{HL2020} G. Hu and J. Li, Uniqueness to inverse source problems in an inhomogeneous medium with a single far-field pattern, SIAM Journal on Mathematical Analysis, 52(2020): 5213-5231.

\bibitem{Isakov2021} V. Isakov and J N. Wang, Uniqueness and increasing stability in electromagnetic inverse source problems, Journal of Differential Equations, 283(2021): 110-135.

\bibitem {Yuan2023} G. Yuan and Y. Zhao, Increasing stability for the inverse source problem in electromagnetic waves with conductivity, Applied Mathematics Letters, 144(2023): 108722.


\bibitem {BaoLu2015} G. Bao, S. Lu, W. Rundell and B. Xu, A recursive algorithm for multifrequency acoustic inverse source problems, SIAM Journal on Numerical Analysis, 53(2015): 1608-1628.


\bibitem{Wang2018}G. Wang, F. Ma, Y. Guo and J. Li, Solving the multi-frequency electromagnetic inverse source problem by the Fourier method, Journal of Differential Equations, 265(2018): 417-443.

\bibitem{Wang2019} X. Wang, M. Song, Y. Guo, H. Li and H. Liu, Fourier method for identifying
electromagnetic source with multi-frequency far field data, Journal of Computational and Applied
Mathematics, 358(2019): 279-292.

\bibitem{GS} R. Griesmaire and C. Schmiedecke, A Factorization method for multifrequency inverse source problem with sparse far-field measurements, SIAM Journal on Imaging Sciences, 10 (2017): 2119-2139.

\bibitem {Griesmaier2018} R. Griesmaier and R K. Mishra and C. Schmiedecke, Inverse Source Problems for Maxwell's Equations and the Windowed Fourier Transform, SIAM Journal on Scientific Computing 40(2018): A1204-A1223.


\bibitem{JL2019} X. Ji and X. Liu, Inverse electromagnetic source scattering problems with multifrequency sparse phased and phaseless far field data, SIAM Journal on Scientific Computing, 41(2019): B1368-B1388.

\bibitem{LiLiu2023} J. Li and X. Liu, Reconstruction of multiscale electromagnetic sources from multifrequency electric far field patterns at sparse observation directions, Multiscale Modeling \& Simulation, 21(2023): 753-775.

\bibitem{Harris2024} I. Harris, T. Le and D. Nguyen, A direct reconstruction method for radiating sources in Maxwell’s equations with single-frequency data, Inverse Problems, 41(2024): 015003.

\bibitem{Liu2024} X. Liu and Q. Shi, The high resolution sampling methods for acoustic sources from multi-frequency far field patterns at sparse observation directions, arxiv:2408.10829, 2024.

\bibitem{Liu2025} X. Liu and Q. Shi, A quantitative sampling method for elastic and electromagnetic sources, Journal of Computational Physics (2025): 114251.

\bibitem{AHLS} A. Alzaalig, G. Hu, X. Liu and J. Sun, Fast acoustic source imaging using multi-frequency sparse data, Inverse Problems, 36 (2020): 025009.

\bibitem{GGH} H. Guo and G. Hu, Inverse wave-number-dependent source problems for the Helmholtz equation, SIAM Journal on Numerical Analysis, 62(2024): 1372-1393.

\bibitem{GHZ} H. Guo, G. Hu and M. Zhao,  Direct sampling method to inverse wave-number-dependent source problems: determination of the support of a stationary source, Inverse Problems, 39(2023): 105008.

\bibitem{Zhao2024} M. Zhao, S. Si and G. Hu, Uniqueness, stability and algorithm for an inverse wave-number-dependent source problem, Inverse Problems, 40(2024): 125019. 

\bibitem{Ma2025} G. Ma, H. Guo and G. Hu, A frequency-domain method to inverse moving source problem with unknown radiating moment, arXiv:2602.04207.

\end{thebibliography}
\end{document}